\documentclass[12pt, reqno]{amsart}
\author[P.~Leonetti]{Paolo Leonetti}
\address{
Universit\`a degli Studi dell'Insubria, via Monte Generoso 71, Varese 21100, Italy}
\email{leonetti.paolo@gmail.com}
\urladdr{\url{https://sites.google.com/site/leonettipaolo/}} 

\makeatletter
\@namedef{subjclassname@2020}{\textup{2020} Mathematics Subject Classification}
\makeatother


\title{Regular matrices of unbounded linear operators}

\usepackage[T1]{fontenc}
\usepackage{amsmath}
\usepackage{amssymb}
\usepackage{amsthm}
\usepackage[left=2.5cm, right=2.5cm, top=3cm]{geometry}
\usepackage{hyperref}
\usepackage{fancyhdr}
\usepackage[inline]{enumitem}
\usepackage{comment}
\usepackage{nicefrac}
\usepackage{bm}
\usepackage{mathrsfs}
\usepackage{graphicx}
\usepackage[utf8]{inputenc}
\usepackage{cancel}
\usepackage{mathtools}

\newcommand{\vertiii}[1]{{\left\vert\kern-0.25ex\left\vert\kern-0.25ex\left\vert #1 
    \right\vert\kern-0.25ex\right\vert\kern-0.25ex\right\vert}}
		
\AtBeginDocument{%
   \def\MR#1{}
}

\newtheorem{thm}{Theorem}[section]
\newtheorem{cor}[thm]{Corollary}
\newtheorem{lem}[thm]{Lemma}
\newtheorem{prop}[thm]{Proposition}

\theoremstyle{definition} 
\newtheorem{defi}[thm]{Definition}
\let\olddefi\defi
\renewcommand{\defi}{\olddefi\normalfont}
\newtheorem{example}[thm]{Example}
\let\oldexample\example
\renewcommand{\example}{\oldexample\normalfont}
\newtheorem{rmk}[thm]{Remark}
\let\oldrmk\rmk
\renewcommand{\rmk}{\oldrmk\normalfont}

\hypersetup{
    pdftitle={Regular matrices of unbounded linear operators},
    pdfauthor={Paolo Leonetti},
    pdfmenubar=false,
    pdffitwindow=true,
    pdfstartview=FitH,
    colorlinks=true,
    linkcolor=blue,
    citecolor=green,
    urlcolor=cyan
}

\providecommand{\MR}[1]{}

\providecommand{\MR}{\relax\ifhmode\unskip\space\fi MR }

\providecommand{\href}[2]{#2}

\begin{document}

\maketitle
\thispagestyle{empty}

\begin{abstract}
\noindent Let $X,Y$ be Banach spaces and fix a 
linear operator $T \in \mathcal{L}(X,Y)$ and ideals $\mathcal{I}, \mathcal{J}$ on the nonnegative integers. We obtain Silverman--Toeplitz type theorems on matrices $A=(A_{n,k}: n,k \in \omega)$ of linear operators in $\mathcal{L}(X,Y)$, so that 
$$
\mathcal{J}\text{-}\lim A\bm{x}=T(\hspace{.2mm}\mathcal{I}\text{-}\lim \bm{x})
$$
for every $X$-valued sequence $\bm{x}=(x_0,x_1,\ldots)$ which is $\mathcal{I}$-convergent [and bounded]. 
This allows us to establish the relationship between the classical Silverman--Toeplitz characterization of regular matrices and its multidimensional analogue for double sequences, its variant for matrices of linear operators, and the recent version (for the scalar case) in the context of ideal convergence. 
As byproducts, we obtain characterizations of several matrix classes and a generalization of the classical Hahn--Schur theorem. 
In the proofs we will use an ideal version of the Banach--Steinhaus theorem which has been recently obtained by De Bondt and Vernaeve in 
[J.~Math.~Anal.~Appl.~\textbf{495} (2021)]. 
\end{abstract}

\section{Introduction}

An infinite matrix with real entries $A=(a_{n,k})$ is said to be \emph{regular} if it transforms convergent sequences into convergent sequences and preserves the corresponding limits (details will be given in Section \ref{sec:mainresults}). 
A classical result due to Silverman--Toeplitz provides necessary and sufficient conditions, depending only on the entries of $A$, which characterize the class of regular matrices, see e.g. \cite[Theorem 2.3.7]{MR1817226}: 
\begin{thm}\label{thm:SilvermanToeplizsimple}
An infinite real matrix $A=(a_{n,k})$ is regular if and only if\textup{:}
\begin{enumerate}[label={\rm (\roman*)}]
\item $\sup_n \sum_k|a_{n,k}|<\infty$\textup{;}
\item $\lim_n\sum_ka_{n,k}=1$\textup{;}
\item $\lim_na_{n,k}=0$ for all $k$\textup{.}
\end{enumerate}
\end{thm}

Several extensions and analogues 
of the characterization above 
can be found in the literature. 
First, a \textquotedblleft multidimensional\textquotedblright\,version of Theorem \ref{thm:SilvermanToeplizsimple} for double sequences has been proved by George M. 
Robinson \cite{Robinson26} and Hugh J. Hamilton \cite{MR1545904}.  
Second, in 1950 Abraham Robinson \cite{MR37371} proved the operator analogue of Theorem \ref{thm:SilvermanToeplizsimple} replacing each $a_{n,k}$ with a (possibly unbounded) linear operator $A_{n,k}$ acting on a given Banach space, cf. Theorem \ref{thm:originalIJREGULAR} below. 
Third, on a different direction, the author and Jeff Connor \cite{ConnorLeo} recently studied the ideal/filter version of the notion of regularity in the scalar case and proved that the analogue of Theorem \ref{thm:SilvermanToeplizsimple} holds in several, but not all, cases. 

The aim of this work is to provide a unifying framework which allows to shed light on the relationship between all the above results, to extend the latter ones, and to obtain, as a byproduct of the employed methods, several related characterizations.  
This will require us to deal with the theory of infinite matrices of linear operators and to prove a certain number of intermediate lemmas. 
Most results are formulated in the context of ideal convergence. We remark that this choice is not done for the sake of generality: indeed, for the above purposes, we will exploit the simple facts that $c(\mathcal{I}) \cap \ell_\infty$ is $c$ if $\mathcal{I}=\mathrm{Fin}$ and equals to $\ell_\infty$ if $\mathcal{I}$ is maximal, that the Pringsheim convergence of double sequences coincides with $\mathcal{I}$-convergence for a suitable ideal $\mathcal{I}$, etc. 
An additional motivation comes from the fact that the study of ideals on countable sets and their representability 
may have some relevant potential for the study of the geometry of Banach spaces, see e.g. \cite{MR4124855, MR3436368, OrhanADDED, KwelaLeonetti, LeonettiCaprio}. 

Informally, we provide an operator version of the characterization of regular matrices in the context of ideal convergence, together with some sufficient conditions which allow for several substantial simplifications.  
The results depend on the boundedness assumption on the sequence spaces in the domain and/or codomain of such matrices. 
In addition, we provide a characterization of the matrix classes 
$(\ell_\infty, c_0(\mathcal{J})\cap \ell_\infty)$, $(c(\mathcal{I}), c_0(\mathcal{J})\cap \ell_\infty)$ $(\ell_\infty, \ell_\infty(\mathcal{J}))$, and $(c,\ell_\infty(\mathcal{J}))$ for certain ideals $\mathcal{I},\mathcal{J}$ on $\omega$, 
see Corollary \ref{cor:llinftyc0}, Theorem \ref{thm:cIc_0bJ}, Theorem \ref{thm:maddoxmain}, and Theorem \ref{thm:maddoxmaincXellinfty}, respectively.  
Lastly, we obtain an ideal version of the Hahn--Schur theorem (which is used to prove that weak and norm convergence coincide on $\ell_1$), see Theorem \ref{thm:hahnschur}. 

The proofs of the main results are given in Section \ref{sec:mainproofs}.




\section{Notations and Main Results}\label{sec:mainresults}

Let $\mathcal{I}$ be an ideal on the nonnegative integers $\omega$, that is, a collection of subsets of $\omega$ which is closed under subsets and finite unions. 
Unless otherwise stated, it is assumed that it contains the collection $\mathrm{Fin}$ of finite sets and it is different from the power set. 
Denote its dual filter by $\mathcal{I}^\star:=\{S\subseteq \omega: S^c \in \mathcal{I}\}$ and define $\mathcal{I}^+:=\{S\subseteq \omega: S\notin \mathcal{I}\}$. 
Among the most important ideals, we find the family of asymptotic density zero sets 
\begin{equation}\label{eq:definitionZ}
\mathcal{Z}:=\left\{S\subseteq \omega: \lim_{n\to \infty} \frac{|S \cap [0,n]|}{n+1}=0\right\}.
\end{equation}
We refer to \cite{MR2777744} 
for a recent survey on ideals and associated filters.

Let $V$ be a real Banach space, and denote its closed unit ball by $B_V$ and its unit sphere by $S_V$. 
Given a sequence $\bm{x}=(x_n)$ taking values in $V$ and an ideal $\mathcal{I}$ on $\omega$, we say that $\bm{x}$ is $\mathcal{I}$\emph{-convergent to} $\eta \in V$, shortened as $\mathcal{I}\text{-}\lim \bm{x}=\eta$ or $\mathcal{I}\text{-}\lim_n x_n=\eta$, if $\{n \in \omega: x_n \notin U\} \in \mathcal{I}$ for all neighborhood $U$ of $\eta$; for the clarify of exposition, all sequences taking values in $V$ will be written in bold. 
Note that $\mathcal{Z}$-convergence is usually called \emph{statistical convergence}, see e.g. \cite{MR1181163}. 
%
As remarked in \cite[Example 3.4]{MR3671266}, the notion of $\mathcal{I}$-convergence include the well-known uniform, Pringsheim, and Hardy convergences for double sequences. 
In addition, if $\bm{y}$ is a real nonnegative sequence, we write $\mathcal{I}\text{-}\limsup \bm{y}:=\inf\{r \in \mathbf{R}\cup \{\infty\}: \{n \in \omega: y_n\ge r\}\in \mathcal{I}\}$.

Now, define the following sequence spaces 
\begin{displaymath}
\begin{split}
\ell_\infty(V)&:=\left\{\bm{x} \in V^\omega: \|\bm{x}\|<\infty\right\},\\
\ell_\infty(V,\mathcal{I})&:=\left\{\bm{x} \in V^\omega: \mathcal{I}\text{-}\limsup\nolimits_n\|x_n\|<\infty\right\},\\
c(V,\mathcal{I})&:=\left\{\bm{x} \in V^\omega: \mathcal{I}\text{-}\lim \bm{x}=\eta \text{ for some }\eta \in V\right\},\\
c_0(V,\mathcal{I})&:=\left\{\bm{x} \in V^\omega: \mathcal{I}\text{-}\lim \bm{x}=0\right\},\\
c_{00}(V,\mathcal{I})&:=\left\{\bm{x} \in V^\omega: \mathrm{supp}\,\bm{x} \in \mathcal{I}\right\},
\end{split}
\end{displaymath}
where $\|\bm{x}\|:=\sup_n\|x_n\|$ stands for the supremum norm and $\mathrm{supp}\,\bm{x}$ for the support $\{n \in \omega: x_n\neq 0\}$. 
Clearly $\ell_\infty(V)=\ell_\infty(V, \mathrm{Fin})$; 
sequences in $\ell_\infty(V,\mathcal{I})$ are usually called $\mathcal{I}$-bounded. 
If $V=\mathbf{R}$ and $\mathcal{I}=\mathrm{Fin}$, the above sequence spaces correspond to the usual $\ell_\infty$, $c$, $c_0$, and $c_{00}$, respectively. 
Every subspace of $\ell_\infty(V)$ will be endowed with the supremum norm. 
It is clear that $c_{00}(V,\mathcal{I})\subseteq c_{0}(V,\mathcal{I})\subseteq c(V,\mathcal{I})\subseteq \ell_\infty(V,\mathcal{I})$. However, unless $\mathcal{I}=\mathrm{Fin}$ or $V=\{0\}$, $c_{00}(V,\mathcal{I})$ is not contained in $\ell_\infty(V)$. 
Hence it makes sense to define the subspace 
$$
c^b(V,\mathcal{I}):=c(V,\mathcal{I})\cap \ell_\infty(V), 
$$
and, similarly, $c_{0}^b(V,\mathcal{I})$ and $c_{00}^b(V,\mathcal{I})$. 
The symbol $V$ will be removed from the notation if it is understood from the context so that, e.g., $c_{00}(\mathcal{I})=c_{00}(V, \mathcal{I})$. Similarly, we may remove $\mathcal{I}$ in the case $\mathcal{I}=\mathrm{Fin}$. 

At this point, let $X,Y$ be two Banach space and denote by $\mathcal{L}(X,Y)$ and $\mathcal{B}(X,Y)$ the vector spaces of linear operators from $X$ to $Y$ and its subspace of bounded linear operators, respectively. 
We assume that $\mathcal{L}(X,Y)$ and all its subspaces are endowed with the strong operator topology, so that a sequence $(T_n)$ of linear operators in $\mathcal{L}(X,Y)$ converges to $T \in \mathcal{L}(X,Y)$ if and only if $(T_nx)$ is convergent in the norm of $Y$ to $Tx$ for all $x \in X$.

Let $A=(A_{n,k}: n,k \in \omega)$ be an infinite matrix of linear operators $A_{n,k}\in \mathcal{L}(X,Y)$. 
Moreover, 
for each $n \in \omega$ and $E\subseteq \omega$, let us write 
$$
A_{n,E}:=(A_{n,k}: k \in E)
$$
and $A_{n,\ge k}:=A_{n,\{k,k+1,\ldots\}}$ for the $k$th tail of the $n$th row of $A$. 
In particular, $A_{n,\omega}$ is the $n$th row of $A$ (and use an analogue notation for a sequence $(T_k)$ of operators so that, for instance, $T_{\ge 2}=(T_2,T_3,\ldots)$). 
For each $n \in \omega$ and $E\subseteq \omega$, define the the group norm 
$$
\|A_{n,E}\|:=
\sup\left\{\left\|\sum\nolimits_{k\in F}A_{n,k}x_k\right\|: F\subseteq E \text{ is finite and each }x_k \in B_X\right\},
$$
cf. \cite{MR0447877, MR568707} 
(in fact, \emph{every} $x_k$ can be chosen on the unit sphere $S_X$: this depends on the fact that, given distinct $a,b \in X$, the function $f: [0,1] \to \mathbf{R}$ defined by $f(t):=\|a+t(b-a)\|$ has a point of maximum in $t=0$ or in $t=1$: indeed, the segment $\{a+t(b-a): t \in [0,1]\}$ is contained in the closed ball with center $0$ and radius $\max\{\|a\|,\|b\|\}$, which is convex). 
Note that the value $\|A_{n,E}\|$ is possibly not finite. 
In addition, if $X=\mathbf{R}$ and $A_{n,E}$ is represented by the real sequence $(a_{n,k}: k \in E)$ then $\|A_{n,E}\|=\sum_{k \in E}|a_{n,k}|$. 

Given an $X$-valued sequence $\bm{x}=(x_n)$, let $A\bm{x}$ be its $A$-transform, that is, the sequence $A\bm{x}:=(A_n\bm{x}: n \in \omega)$ where
$$
\forall n \in \omega, \quad 
A_n\bm{x}:=\sum\nolimits_{k}A_{n,k}x_k, 
$$
provided that each series is convergent in the norm of $Y$. Accordingly, let $\mathrm{dom}(A)$ be the domain of $A$, that is, the family of those sequences $\bm{x}$ such that $A\bm{x}$ is well defined. 
For each sequence subspace $\mathscr{A}\subseteq X^\omega$ and $\mathscr{B}\subseteq Y^\omega$, let $(\mathscr{A}, \mathscr{B})$ be the set of matrices $A=(A_{n,k})$ of 
(not necessarily bounded) 
linear operators in $\mathcal{L}(X,Y)$ such that 
$$
\mathscr{A} \subseteq \mathrm{dom}(A)
\,\,\,\text{ and }\,\,\,
A\bm{x} \in \mathscr{B} \,\text{ for all }\bm{x} \in \mathscr{A}.
$$ 
We refer to \cite{MR568707} for the theory of infinite matrices of operators. 
In the scalar case, the relationship between summability and ideal convergence has been recently studied in \cite{Filipow18}. 

\subsection{Bounded to Bounded case}\label{subsec:boundedbounded}
The main definition of this work follows:
\begin{defi}\label{def:mainIJregular}
Let $\mathcal{I}, \mathcal{J}$ be ideals on $\omega$ and fix $T \in \mathcal{L}(X,Y)$. 
Then a matrix $A=(A_{n,k})$ of linear operators in $\mathcal{L}(X,Y)$ is said to be $(\mathcal{I}, \mathcal{J})$\emph{-regular with respect to $T$} if
$$
A \in (c^b(X,\mathcal{I}),c^b(Y,\mathcal{J})) 
\,\,\, \text{ and }\,\,\,
\mathcal{J}\text{-}\lim A\bm{x}=T(\,\mathcal{I}\text{-}\lim \bm{x})
\,\text{ for all }\bm{x} \in c^b(X,\mathcal{I}).
$$
\end{defi}
If $T$ is the identity operator $I$ on $X$ (namely, $Ix=x$ for all $x \in X$), we will simply say that $A$ is $(\mathcal{I},\mathcal{J})$\emph{-regular}. 
Note that if $\mathcal{I}=\mathcal{J}=\mathrm{Fin}$, $X=Y=\mathbf{R}$, 
and $T=I$, 
then Definition \ref{def:mainIJregular} corresponds to the ordinary regular matrices. 

The following result, essentially due to Robinson \cite[Theorem VII]{MR37371}, is the (unbounded) operator version of Theorem \ref{thm:SilvermanToeplizsimple},  
cf. also \cite[Theorem 1]{MR361524}.  
\begin{thm}\label{thm:originalIJREGULAR}
Fix a linear operator $T \in \mathcal{B}(X,Y)$, where $X,Y$ are Banach spaces. 
Then a matrix $A=(A_{n,k})$ of linear operators in $\mathcal{L}(X,Y)$ is $(\mathrm{Fin}, \mathrm{Fin})$-regular with respect to $T$ if and only if there exists $k_0 \in \omega$ such that\textup{:} 
\begin{enumerate}[label={\rm (\textsc{S}\arabic{*})}]
\item \label{item:S1} $\sup_n\|A_{n,\ge k_0}\|<\infty$\textup{;} 
\item \label{item:S2} $\lim_n \sum_{k}A_{n,k}=T$\textup{;}
\item \label{item:S3} $\lim_n A_{n,k}=0$ for all $k \in \omega$\textup{.}
\end{enumerate}
\end{thm}
A variant for continuous linear operators between Fr\'{e}chet spaces has been proved by Ramanujan in \cite{MR216203}, cf. also \cite[Corollary 6]{MR1936721}. 

Notice that \ref{item:S3} can be rewritten as $\lim_n \sum_{k \in E}A_{n,k}x_k=0$ for all sequences $\bm{x} \in \ell_\infty(X)$ and $E \in \mathrm{Fin}$, which is also equivalent to $A  \in (c_{00}(X), c_0(Y))$. 
\begin{rmk}\label{rmk:evilgroupnorm}
Condition \ref{item:S3} may look equivalent also to: 
\begin{enumerate}[label={\rm (\textsc{S}\arabic{*})}]
\setcounter{enumi}{3}
\item \label{item:S3sharp} $\lim_n \|A_{n,k}\|=0$ for all $k \in \omega$. 
\end{enumerate}
This is correct if $X$ is finite dimensional, cf. Lemma \ref{lem:finitedimensionalnormkjfdhgd} below. 
However, \ref{item:S3sharp} is strictly stronger in general. For, set $X=\ell_2=\{x \in \ell_\infty: \sum_tx_t^2<\infty\}$, and define $A_{n,0}x=(0,\ldots,0,x_{n+1},x_{n+2},\ldots)$ for all $n \in \omega$ and $x \in \ell_2$, and $A_{n,k}=0$ whenever $k>0$. 
Then $\|A_{n,0}\|=1$ for all $n\in \omega$, and $\lim_n A_{n,0}x=0$ for all $x \in \ell_2$. 
\end{rmk}


\begin{rmk}\label{rmk:regulardoesnotimplylinftylinfty}
Another difference from the finite dimensional case is that a $(\mathrm{Fin}, \mathrm{Fin})$-regular matrix 
does not necessarily belong to $(\ell_\infty(X),\ell_\infty(Y))$. 
For, suppose that $X=Y=\ell_2$, 
and let $e_k$ be the $k$th unit vector of $X$ for each $k \in \omega$. 
Building on the above example, consider the matrix $A=(A_{n,k})$ of linear operators in $\mathcal{L}(\ell_2,\ell_2)$ such that $A=\mathrm{Id}+B$, where $\mathrm{Id}$ is the identity matrix and 
\begin{displaymath}
\forall n,k \in \omega, \forall x \in \ell_2, \quad 
B_{n,k}(x):=
\begin{cases}
\,(0,\ldots,0,x_{n+1},x_{n+2},\ldots)\,\,& \text{if }k=0;\\
\,-x_{n+k}e_{n+k} & \text{if }k>0.
\end{cases}
\end{displaymath}
Then $A$ satisfies conditions \ref{item:S1}-\ref{item:S3} with $T=I$ and $k_0=0$, hence by Theorem \ref{thm:SilvermanToeplizsimple} $A$ is a $(\mathrm{Fin}, \mathrm{Fin})$-regular matrix. 
However, the sequence $\bm{x}:=(e_0,e_1,\ldots) \in \ell_\infty(\ell_2)$ does not belong to $\mathrm{dom}(A)$, indeed $A_0\bm{x}=e_0-e_1-e_2-\cdots$ is not norm convergent in $\ell_2$. 
Therefore $A\notin (\ell_\infty(X),\ell_\infty(Y))$, cf. Theorem \ref{thm:maddoxmain} below. 
\end{rmk}

Our first main result, which corresponds to the operator version of \cite[Theorem 1.2]{ConnorLeo}, follows. 
\begin{thm}\label{main:IJREGULAR}
Fix a linear operator $T \in \mathcal{L}(X,Y)$, where $X,Y$ are Banach spaces. 
Let also $\mathcal{I}, \mathcal{J}$ be ideals on $\omega$. 
Then a matrix $A=(A_{n,k})$ of linear operators in $\mathcal{L}(X,Y)$ is $(\mathcal{I}, \mathcal{J})$-regular with respect to $T$ if and only if 
%
%
%
there exists $k_0 \in \omega$ such that\textup{:} 
\begin{enumerate}[label={\rm (\textsc{T}\arabic{*})}]
\item \label{item:T1} $\sup_n\|A_{n,\ge k_0}\|<\infty$\textup{;}
\item \label{item:T2} 
$\sup_n\|A_{n,k}x\|<\infty$ for all $x \in X$ and $k<k_0$\textup{;}
\item \label{item:T3} $\sum_kA_{n,k}x_k$ converges in the norm of $Y$ for all $\bm{x} \in c^b(X,\mathcal{I})$ and $n \in \omega$\textup{;}
\item \label{item:T4} $\mathcal{J}\text{-}\lim_n \sum_{k}A_{n,k}=T$\textup{;}
\item \label{item:T5} $A \in (c_{00}^b(X,\mathcal{I}), c_0(Y,\mathcal{J}))$\textup{.} 
\end{enumerate}
In addition, if each $A_{n,k}$ is bounded, it is possible to choose $k_0=0$. 
\end{thm}

For the sake of clarity, 
condition \ref{item:T4} means that $A_n(x,x,\ldots)$ is norm convergent for all $n \in \omega$ and $x \in X$ (which is weaker than \ref{item:T3}) and, in addition, $\mathcal{J}\text{-}\lim_n A_n(x,x,\ldots)=Tx$ for all $x \in X$. 
Lastly, condition \ref{item:T5} can be rephrased as: $A_n\bm{x}$ is norm convergent for all $n \in \omega$ and all bounded sequences $\bm{x}$ supported on $\mathcal{I}$ and, for such sequences, $\mathcal{J}\text{-}\lim A\bm{x}=0$. 
\begin{rmk}\label{rmk:conditionT5}
Note that \ref{item:T5} could be replaced also with the stronger condition: 
\begin{enumerate}[label={\rm (\textsc{T}\arabic{*}$^\prime$)}]
\setcounter{enumi}{4}
\item \label{item:T5prime} $A \in (c_{00}^b(X,\mathcal{I}), c_0^b(Y,\mathcal{J}))$.
\end{enumerate}
Indeed by Definition \ref{def:mainIJregular} the transformed sequence $A\bm{x}$ is necessarily bounded for all $\bm{x} \in c^b(X,\mathcal{I})$. 
Also, the latter condition \ref{item:T5prime} would imply automatically \ref{item:T2}. 
Therefore, $A$ is $(\mathcal{I},\mathcal{J})$-regular with respect to $T$ if and only if \ref{item:T1}, \ref{item:T3}, \ref{item:T4}, and \ref{item:T5prime} hold. 

However, we chose to state it in the former version for two reasons. 
First, if each $A_{n,k}$ is bounded, then \ref{item:T2} is void so that our characterization holds with the weaker condition \ref{item:T5}. 
Second, most importantly, condition \ref{item:T5} will be used also in the unbounded analogue given in Theorem \ref{main:IJREGULARboundedUnbounded} below: this allows to highlight the differences between the two cases.
\end{rmk}

Even if Theorem \ref{main:IJREGULAR} may look quite complicated, the reader should keep in mind that it deals with (possibly unbounded) linear operators and general ideal/filter convergence. 
We are going to see that, in some special circumstances, it may be considerably simplified because either some of the conditions \ref{item:T1}-\ref{item:T5} are automatically satisfied or the latter ones collapse to simpler properties (in particular, recovering the classical ones). 
%
Several related results 
may be found in the literature in the case $X=Y$ equal to $\mathbf{R}$ or $\mathbf{C}$, $T$ equals to the identity operator $I$ or the zero operator, and $\mathcal{I},\mathcal{J}$ being certain $F_{\sigma\delta}$-ideals (where ideals are regarded as subsets of the Cantor space $\{0,1\}^\omega$), see e.g. \cite{MR3511151, 
MR1963462, MR3671266, MR3911031, MR1433948}. 

We remark also that, if $T$ is not bounded, then an $(\mathcal{I},\mathcal{J})$-regular matrix $A$ with respect to $T$ may not exist: indeed, if each $A_{n,k}$ is bounded and $\mathcal{J}=\mathrm{Fin}$, condition \ref{item:T4} and the Banach--Steinhaus theorem imply that $T$ is necessarily bounded.

In the case that $T=0$, we obtain the following immediate consequence, cf. also Corollary \ref{cor:llinftyc0} below for the finite dimensional case with $\mathcal{I}$ maximal. 
\begin{cor}\label{cor:ctoc0}
Let $X,Y$ be Banach spaces, and let also $\mathcal{I}, \mathcal{J}$ be ideals on $\omega$. 
Then a matrix $A=(A_{n,k})$ of linear operators in $\mathcal{L}(X,Y)$ belongs to $(c^b(X,\mathcal{I}), c^b_0(Y,\mathcal{J}))$ if and only if there exists $k_0 \in \omega$ such that \ref{item:T1}-\ref{item:T5} hold, with $T=0$. 

In addition, if each $A_{n,k}$ is bounded, it is possible to choose $k_0=0$. 
\end{cor}

It will be useful to define also the following properties: 
\begin{enumerate}[label={\rm (\textsc{T}\arabic{*}$^{\natural}$)}]
\setcounter{enumi}{2}
\item \label{item:T3natural} $\lim_k \|A_{n,\ge k}\|=0$ for all $n\in \omega$; 
\end{enumerate}
\begin{enumerate}[label={\rm (\textsc{T}\arabic{*}$^\flat$)}]
\setcounter{enumi}{5}
\item \label{item:T6flat} $\mathcal{J}\text{-}\lim_n \|A_{n,k}\|=0$ for all $k\in \omega$. 
\end{enumerate}
It is clear that \ref{item:S3sharp} corresponds to \ref{item:T6flat} in the case $\mathcal{J}=\mathrm{Fin}$. 
Some implications between the above-mentioned conditions 
are collected below. 
\begin{prop}\label{prop:implications}
With the same hypothesis of Theorem \ref{main:IJREGULAR}, the following hold\textup{:}
\begin{enumerate}[label={\rm (\roman*)}]
\item \label{item:1simplification} If $\mathcal{I}=\mathrm{Fin}$ then \ref{item:T1} and \ref{item:T4} imply \ref{item:T3}\textup{;}
\item \label{item:2simplification} If $\mathcal{J}=\mathrm{Fin}$ then \ref{item:T5} implies \ref{item:T2}\textup{;}
\item \label{item:3simplification} \ref{item:T1} and \ref{item:T3natural} imply \ref{item:T3}\textup{;}
\item \label{item:4simplification} If each $A_{n,k}$ is bounded, it is possible to choose $k_0=0$, hence \ref{item:T2} holds\textup{;}
\item \label{item:5simplification} If $\mathrm{dim}(X)<\infty$ each $A_{n,k}$ is bounded. Moreover, \ref{item:T1} implies \ref{item:T3natural}\textup{;}
\item \label{item:6simplification} If $\mathrm{dim}(X)<\infty$ and $A$ is $(\mathcal{I}, \mathcal{J})$-regular with respect to $T$, then \ref{item:T6flat} holds\textup{;}
\item \label{item:7simplification} If each $A_{n,k}$ is a multiple of  $A_0 \in \mathcal{L}(X,Y)$, then \ref{item:T1} implies \ref{item:T3natural} and $A_0\in \mathcal{B}(X,Y)$\textup{;}
\item \label{item:8simplification} If each $A_{n,k}$ is a multiple of  $A_0 \in \mathcal{L}(X,Y)$, then \ref{item:T1} and \ref{item:T5} imply \ref{item:T6flat}\textup{.}
\end{enumerate}
\end{prop}


It is immediate to check that Theorem \ref{thm:originalIJREGULAR} comes as a corollary, putting together Theorem \ref{main:IJREGULAR} and Proposition \ref{prop:implications}.\ref{item:1simplification} and \ref{item:2simplification}. 

However, the 
usefulness of a characterization of $(\mathcal{I}, \mathcal{J})$-regular matrices with respect to $T$ comes from the practical easiness to check whether conditions \ref{item:T1}-\ref{item:T5} hold together.  
Taking into account the implications given in Proposition \ref{prop:implications}, it is evident that \ref{item:T5} is the most demanding in this direction. 
Hence it makes sense to search for sufficient conditions which allow us to simplify it. 
In the same spirit of \cite[Theorem 1.3]{ConnorLeo}, which studies the case $X=Y=\mathbf{R}$ and $T=I$, we obtain 
characterizations of 
such matrices 
which avoid condition \ref{item:T5}. 
We will need the new and much \textquotedblleft easier\textquotedblright\,condition:
\begin{enumerate}[label={\rm (\textsc{T}\arabic{*})}]
\setcounter{enumi}{5}
\item \label{item:T6} $\mathcal{J}\text{-}\lim_n \|A_{n,E}\|=0$ for all $E \in \mathcal{I}$. 
\end{enumerate} 
Directly by the definition of group norm, it is clear that \ref{item:T6} implies \ref{item:T5} (and also \ref{item:T6flat}). 
This means that we are allowed to replace \ref{item:T5} with the stronger condition \ref{item:T6} provided that the latter is satisfied for matrices $A$ which are $(\mathcal{I}, \mathcal{J})$-regular with respect to $T$, possibly under some additional constraints. 

\begin{thm}\label{thm:JfinIJregularmain}
With the same hypotheses of Theorem \ref{main:IJREGULAR}, suppose, in addition, that $\mathcal{J}$ is countably generated and that conditions \ref{item:T3natural} and \ref{item:T6flat} hold. 

Then $A$ is $(\mathcal{I}, \mathcal{J})$-regular with respect to $T$ if and only if there exists $k_0 \in \omega$ such that \ref{item:T1}, \ref{item:T4}, and \ref{item:T6} hold. 
In addition, if each $A_{n,k}$ is bounded, it is possible to choose $k_0=0$. 
\end{thm}

On a similar direction, recall that, if $X,Y$ are vector lattices, then a linear operator $T \in \mathcal{L}(X,Y)$ is said to be \emph{positive} if 
$Tx\ge 0$ 
whenever $x\ge 0$. 
In addition, a Banach space $V$ is called an \emph{AM-space} if $V$ is also a vector lattice such that $0\le x\le y$ implies $\|x\| \le \|y\|$, and $\|x \vee y\|=\max\{\|x\|,\|y\|\}$ for all $x,y \ge 0$; 
we say that  $e \in V$ is an \emph{order unit} if, for all $x \in V$ there exists $n \in \omega$ such that $-ne \le x\le ne$. 
Accordingly, if $V\neq \{0\}$, then necessarily $e>0$. 
Examples of AM-spaces with order units include $\ell_\infty$ and $C(K)$ spaces, for some compact Hausdorff space $K$. 
We refer to \cite{MR2011364, MR2262133} for the underlying theory on vector lattices. 

\begin{thm}\label{thm:AMspaceregularmain}
With the same hypotheses of Theorem \ref{main:IJREGULAR}, suppose, in addition, that $X$ is an AM-space with order unit $e$, $Y$ is a Banach lattice, each $A_{n,k}$ is a positive linear operator, and that  condition \ref{item:T3natural} holds. 

Then $A$ is $(\mathcal{I}, \mathcal{J})$-regular with respect to $T$ if and only if \ref{item:T1}, \ref{item:T4}, and \ref{item:T6} hold with $k_0=0$. 
\end{thm}

In the finite dimensional case, everything is simpler. Indeed, suppose that $X=\mathbf{R}^d$ and $Y=\mathbf{R}^m$, for some integers $d,m\ge 1$. Then each linear operator $A_{n,k}$ is represented by the real matrix $[\,a_{n,k}(i,j): 1\le i\le m, 1\le j\le d\,]$ and $T$ is represented by the real matrix $[\,t(i,j): 1\le i\le m, 1\le j\le d\,]$ . 
\begin{cor}\label{cor:finitedimensionmain}
With the same hypotheses of Theorem \ref{main:IJREGULAR}, suppose that $X=\mathbf{R}^d$, $Y=\mathbf{R}^m$, and that $\mathcal{I}=\mathrm{Fin}$ or $\mathcal{J}$ is countably generated or $a_{n,k}(i,j)\ge 0$ for all $1\le i \le m$, $1\le j\le d$, and $n,k \in \omega$. 
Then $A$ is $(\mathcal{I}, \mathcal{J})$-regular with respect to $T$ if and only if\textup{:}
\begin{enumerate}[label={\rm (\textsc{F}\arabic{*})}]
\item \label{item:F1} $\sup_n \sum_k\sum_{i,j}\left|a_{n,k}(i,j)\right|<\infty$\textup{;}
\item \label{item:F4} $\mathcal{J}\text{-}\lim_n \sum_k a_{n,k}(i,j)=t(i,j)$ for all $1\le i \le m$ and $1\le j\le d$\textup{;}
\item \label{item:F6} $\mathcal{J}\text{-}\lim_n \sum_{k\in E}\sum_{i,j}\left|a_{n,k}(i,j)\right|=0$ for all $E \in \mathcal{I}$\textup{.}
\end{enumerate}
\end{cor}
It is remarkable that the \textquotedblleft easier\textquotedblright\,characterization with condition \ref{item:F6} does \emph{not} hold uniformly for all ideals $\mathcal{I}, \mathcal{J}$: indeed, it has been proved in \cite[Theorem 1.4]{ConnorLeo} that, even in the simplest case $X=Y=\mathbf{R}$ and $T=I$, there exists a $(\mathcal{Z}, \mathcal{Z})$-regular matrix which does not satisfy \ref{item:F6}, where $\mathcal{Z}$ is the asymptotic density zero ideal defined in \eqref{eq:definitionZ}. 
In addition, condition \ref{item:F6} can be simplified if $T$ is the zero operator and $\mathcal{I}$ a maximal ideal:
\begin{cor}\label{cor:llinftyc0}
With the same hypotheses of Corollary \ref{cor:finitedimensionmain}, $A \in (\ell_\infty(\mathbf{R}^d), c_0^b(\mathbf{R}^m, \mathcal{J}))$ if and only if condition \ref{item:F1} holds, together with:
\begin{enumerate}[label={\rm (\textsc{F}\arabic{*}$^\prime$)}]
\setcounter{enumi}{2}
\item \label{item:F6prime} $\mathcal{J}\text{-}\lim_n \sum_{k}\sum_{i,j}\left|a_{n,k}(i,j)\right|=0$.
\end{enumerate}
\end{cor}
This provides a generalization of \cite[Lemma 3.2]{MR2209588} in the case $d=m=1$, $T=0$, and $\mathcal{J}$ equal to the countably generated ideal $\mathcal{I}_{\mathrm{P}}$ defined below in \eqref{eq:definitionIP}.

A similar result can be obtained if each $A_{n,k}$ is a multiple of a given linear operator:
\begin{cor}\label{cor:eachAnkmultiple}
With the same hypotheses of Theorem \ref{main:IJREGULAR}, suppose that each $A_{n,k}$ is a multiple of a nonzero $A_0 \in \mathcal{L}(X,Y)$, so that 
$
A_{n,k}=a_{n,k}A_0
$ 
for all $n,k \in \omega$. 
In addition, assume that $\mathcal{I}=\mathrm{Fin}$, or $\mathcal{J}$ is countably generated, or that $X$ is an AM-space with order unit $e$, $Y$ is a Banach lattice, and $a_{n,k}\ge 0$ for all $n,k \in \omega$. 

Then $A$ is $(\mathcal{I}, \mathcal{J})$-regular with respect to $T$ if and only if\textup{:}
\begin{enumerate}[label={\rm (\textsc{M}\arabic{*})}]
\item \label{item:M0} $A_0 \in \mathcal{B}(X,Y)$\textup{;}
\item \label{item:M1} $\sup_n \sum_k\left|a_{n,k}\right|<\infty$\textup{;}
\item \label{item:M4} $T=\kappa A_0$, with $\kappa=\mathcal{J}\text{-}\lim_n \sum_k a_{n,k}$\textup{;}
\item \label{item:M6} $\mathcal{J}\text{-}\lim_n \sum_{k\in E}\left|a_{n,k}\right|=0$ for all $E \in \mathcal{I}$\textup{.}
\end{enumerate}
\end{cor}
In particular, under the hypotheses of Corollary \ref{cor:eachAnkmultiple}, 
if $A$ is $(\mathcal{I},\mathcal{J})$-regular with respect to a linear operator $T$, then $T$ is necessarily bounded. 


In the following sections, we obtain the analogues of Theorem \ref{main:IJREGULAR} where we replace the bounded sequence spaces $c^b(X,\mathcal{I})$ and $c^b(Y,\mathcal{J})$ of Definition \ref{def:mainIJregular} with their unbounded versions $c(X,\mathcal{I})$ and $c(Y,\mathcal{J})$. 

\subsection{Bounded to Unbounded case} 

An ideal $\mathcal{J}$ on $\omega$ is said to be a \emph{rapid}$^+$\emph{-ideal} if, 
for every $S \in \mathcal{J}^+$ and $F \in \mathrm{Fin}^+$, 
there exists $S^\prime \subseteq S$ such that $S^\prime \in \mathcal{J}^+$ and $|S^\prime \cap [0,n]|\le |F \cap [0,n]|$ for all $n\in \omega$.
Moreover, $\mathcal{J}$ is called a $P^+$\emph{-ideal} if, for every decreasing sequence $(S_n)$ in $\mathcal{J}^+$, there exists $S \in \mathcal{J}^+$ such that $S\setminus S_n$ is finite for all $n \in \omega$. The class of rapid$^+$ and $P^+$-ideals have been studied also, e.g., in \cite{MR4172859, MR2876731, MR2777744, MR1649074}. 
The ideal $\mathcal{J}$ is said to be \emph{countably generated} if there exists a sequence $(Q_j)$ of subsets of $\omega$ such that $S \in \mathcal{J}$ if and only if $S\subseteq \bigcup_{j \in F}Q_j$ for some $F \in \mathrm{Fin}$. 

Moreover, an ideal $\mathcal{J}$ on $\omega$ is said to be \emph{selective} if, for every decreasing sequence $(S_n)$ in $\mathcal{J}^+$, there exists $S \in \mathcal{J}^+$ such that $S\setminus [0,n] \subseteq S_n$ for all $n \in \omega$, see e.g. \cite[Definition 7.3]{MR2603812}. (They were introduced by Mathias in \cite{MR0491197} under the name of happy families.) It is easy to see, directly from the definitions, that every selective ideal is a rapid$^+$ $P^+$-ideal. More precisely, it is known that an ideal $\mathcal{J}$ on $\omega$ is selective if and only if $\mathcal{J}$ is a $P^+$-ideal and, in addition, for every sequence $(F_n)$ of finite subsets of $\omega$ with $F:=\bigcup_n F_n \in \mathcal{J}^+$, there exists $S \subseteq F$ such that $S \in \mathcal{J}^+$ and $|S \cap F_n|\le 1$ for all $n \in \omega$, see e.g. \cite[Lemma 7.4]{MR2603812}. We refer to \cite[Section 1]{MR3666945} for a list of known examples of selective ideals. 
It is a folklore fact that every $F_\sigma$-ideal is a $P^+$-ideal, see e.g. \cite{MR748847}. 
However, the summable ideal $\mathcal{I}_{1/n}:=\{S\subseteq \omega: \sum_{n \in S}1/(n+1)<\infty\}$ is an $F_\sigma$-ideal which is not rapid$^+$, and $\mathcal{Z}$ is neither a rapid$^+$ nor $P^+$-ideal (hence, they are not selective). 
This does not mean that the topological complexity of selective ideals is low: indeed, under Martin's axiom for countable posets, there exist uncountably many nonisomorphic maximal selective ideals (on the other hand, their existence is not provable in ZFC), see \cite[Section 5.1]{MR4172859}. 

For the next results, we need a slightly stronger version of selectivity: an ideal $\mathcal{J}$ on $\omega$ is called \emph{strongly selective} if, for every decreasing sequence $(S_n)$ in $\mathcal{J}^+$, there exists $S=\{x_n: n \in \omega\} \in \mathcal{J}^+$ such that $x_{n+1} \in S_{x_n}$ for all $n \in \omega$. (Here, we are using $(x_n: n \in \omega)$ for the canonical enumeration of $S$. We are not aware whether this naming has been used somewhere else.) Of course, every strongly selective ideal is also selective. 
Note that every countably generated ideal is strongly selective (hence also rapid$^+$ and $P^+$); in particular, $\mathrm{Fin}$ is a strongly selective ideal, cf. Remark \ref{rmk:farahisomorphiccountablygeneratedideals} below. More generally, it is easy to see that, if $\mathcal{J}$ is an ideal on $\omega$ with the Baire property, then $\mathcal{J}$ is selective if and only if it is strongly selective. Moreover, as pointed out by the referee, maximal strongly selective ideals coincide with maximal selective ideals, thanks to known characterizations of the latter ones (in particular, by the observations above, their existence can shown under Martin's axiom).


Lastly, we need the following weakening of condition \ref{item:T1} (and they coincide if $\mathcal{J}=\mathrm{Fin}$):
\begin{enumerate}[label={\rm (\textsc{T}\arabic{*}$^\flat$)}]
\setcounter{enumi}{0}
\item \label{item:T1flat} There exists $J_0 \in \mathcal{J}^\star$ 
for which $\sup_{n \in J_0}\|A_{n,\ge k_0}\|<\infty$ and, for all $n \in \omega\setminus J_0$, there exists $f(n) \in \omega$ such that $\|A_{n,\ge f(n)}\|<\infty$\textup{.} 
\end{enumerate}
This condition has been suggested by the example given in \cite[Section 4]{ConnorLeo}. 

With these premises, we state the analogue of Theorem \ref{main:IJREGULAR} for the unbounded codomain sequence spaces. 
\begin{thm}\label{main:IJREGULARboundedUnbounded}
Fix a linear operator $T \in \mathcal{L}(X,Y)$, where $X,Y$ are Banach spaces. 
Let also $\mathcal{I},\mathcal{J}$ be an ideal on $\omega$ such that $\mathcal{J}$ is a strongly selective ideal. 

Then a matrix $A=(A_{n,k})$ of linear operators in $\mathcal{L}(X,Y)$ satisfies
\begin{equation}\label{eq:boundedunboundeddefinitiona}
A \in (c^b(X,\mathcal{I}),c(Y,\mathcal{J})) 
\,\,\, \text{ and }\,\,\,
\mathcal{J}\text{-}\lim A\bm{x}=T(\,\mathcal{I}\text{-}\lim \bm{x})
\,\text{ for all }\bm{x} \in c^b(X,\mathcal{I}).
\end{equation}
if and only if 
there exists $k_0 \in \omega$ such that \ref{item:T1flat}, \ref{item:T3}, \ref{item:T4}, and \ref{item:T5} hold.  

In addition, if each $A_{n,k}$ is bounded, it is possible to choose $k_0=f(n)=0$ for all $n\in \omega\setminus J_0$. 
\end{thm}

During the proof, we will need an ideal version of the Banach--Steinhaus theorem which has been recently proved in \cite{MR4172859}, see Theorem \ref{thm:uniformJboundedness} below. 
Interestingly, the latter result provides a characterization of rapid$^+$ $P^+$-ideals \cite[Theorem 5.1]{MR4172859}, which suggests that Theorem \ref{main:IJREGULARboundedUnbounded} cannot be improved with the current techniques. 

We remark that Theorem \ref{main:IJREGULARboundedUnbounded} sheds light on the substantial difference between the classical Silverman--Toeplitz characterization stated in Theorem \ref{thm:SilvermanToeplizsimple} 
and its \textquotedblleft multidimensional\textquotedblright\,analogue proved by Robinson \cite{Robinson26} and Hamilton \cite{MR1545904} for double sequences, namely, the weakening of \ref{item:T1} to \ref{item:T1flat}. 
For, recall that a double sequence $(x_{m,n}: m,n \in \omega)$ has \emph{Pringsheim limit} $\eta \in X$, shortened as $\mathrm{P}\text{-}\lim_{m,n}x_{m,n}=\eta$, if for all $\varepsilon>0$ there exists $k \in \omega$ such that $\|x_{m,n}-\eta\|<\varepsilon$ for all $m,n \ge k$. 
At this point, define the ideal
\begin{equation}\label{eq:definitionIP}
\mathcal{I}_{\mathrm{P}}:=\left\{S\subseteq \omega: \sup\nolimits_{n \in S}\nu_2(n)<\infty\right\},
\end{equation}
where $\nu_2$ is the $2$-adic valution defined by $\nu_2(0):=0$ and $\nu_2(n):=\max\{k \in \omega: 2^k \text{ divides }n\}$ if $n>0$. 
Note that the ideal $\mathcal{I}_{\mathrm{P}}$ is countably generated by the sequence of sets $(Q_t: t \in \omega)$, where $Q_t:=\left\{S\subseteq \omega: \sup\nolimits_{n \in S}\nu_2(n)=t\right\}$ for all $t \in \omega$.  
Hence $\mathcal{I}_{\mathrm{P}}$ is a strongly selective ideal. 
Let also $h: \omega^2\to \omega$ be an arbitrary bijection with the property that $h[\{(m,n) \in \omega^2: \min\{m,n\}=k\}]=Q_k$ for all $k \in \omega$. 
Thus, we obtain 
\begin{equation}\label{eq:equivalencepringhsheim}
\mathrm{P}\text{-}\lim\nolimits_{m,n}x_{m,n}=\eta 
\quad \text{ if and only if }\quad 
\mathcal{I}_{\mathrm{P}}\text{-}\lim\nolimits_n x_{h^{-1}(n)}=\eta,
\end{equation}
as it has been observed in \cite[Section 4.2]{MR3955010}, cf. also \cite{MR3671266}. 
In other words, $\mathcal{I}_{\mathrm{P}}$ is an isomorphic copy on $\omega$ of the ideal on $\omega^2$ generated by vertical lines and horizontal lines, cf. Remark \ref{rmk:farahisomorphiccountablygeneratedideals} below. 
Relying on the equivalence \eqref{eq:equivalencepringhsheim}, the classical definition of \emph{RH-regular matrix} $A$ coincides with \eqref{eq:boundedunboundeddefinitiona} in the case $X=Y=\mathbf{R}$, $T=I$, and $\mathcal{I}=\mathcal{J}=\mathcal{I}_{\mathrm{P}}$. 
With the same notations of Corollary \ref{cor:finitedimensionmain}, we can state the following consequence in the finite dimensional case. 
\begin{cor}\label{cor:RHgeneral}
Suppose that $X=\mathbf{R}^d$, $Y=\mathbf{R}^m$, and let $\mathcal{I}$, $\mathcal{J}$ be ideals on $\omega$ such that 
$\mathcal{J}$ 
is countably generated by a sequence of sets $(Q_t: t \in \omega)$. 

Then a matrix $A$ satisfies \eqref{eq:boundedunboundeddefinitiona} if and only if there exists $t_0 \in \omega$ such that \textup{:}
\begin{enumerate}[label={\rm (\textsc{R}\arabic{*})}]
\item \label{item:R1} $\sup_{n \in \omega\setminus Q_{t_0}} \sum_k\sum_{i,j}\left|a_{n,k}(i,j)\right|<\infty$\textup{;}
\item \label{item:R2} $\sum_k\left|a_{n,k}(i,j)\right|<\infty$ for all $n \in Q_{t_0}$, $1\le i\le m$, and $1\le j\le d$\textup{;}
\item \label{item:R4} $\mathcal{J}\text{-}\lim_n \sum_k a_{n,k}(i,j)=t(i,j)$ for all $1\le i \le m$ and $1\le j\le d$\textup{;}
\item \label{item:R6} $\mathcal{J}\text{-}\lim_n \sum_{k\in E}\sum_{i,j}\left|a_{n,k}(i,j)\right|=0$ for all $E \in \mathcal{I}$\textup{.}
\end{enumerate}
\end{cor}
It is clear that, if also $\mathcal{I}$ is countably generated by a sequence $(E_t: t \in \omega)$, which is the case 
of RH-regular matrices, 
then \ref{item:R6} can be rewritten as: 
\begin{enumerate}[label={\rm (\textsc{R}\arabic{*}$^\prime$)}]
\setcounter{enumi}{3}
\item $\mathcal{J}\text{-}\lim_n \sum_{k\in E_t}\sum_{i,j}\left|a_{n,k}(i,j)\right|=0$ for all $t \in \omega$\textup{.}
\end{enumerate}

Another special instance of Corollary \ref{cor:RHgeneral} has been proved in \cite[Theorem 5]{MR3466559} for the case where $X=Y=\mathbf{R}$, $T=I$, $A$ is a RH-regular matrix with nonnegative real entries, $\mathcal{I}$ is a $P$-ideal (that is, if $(S_n)$ is an increasing sequence in $\mathcal{I}$, there exists $S \in \mathcal{I}$ such that $S_n\setminus S \in \mathrm{Fin}$ for all $n$), and $\mathcal{J}=\mathrm{Fin}$. 

Other consequences of Theorem \ref{main:IJREGULARboundedUnbounded}, in the same vein of the ones given in Section \ref{subsec:boundedbounded}, may be obtained here, and they are left to the reader. 

\begin{rmk}\label{rmk:farahisomorphiccountablygeneratedideals} 
An ideal $\mathcal{I}$ on $\omega$ is countably generated if and only if it is isomorphic to one of the following: 
\begin{enumerate}[label={\rm (\roman*)}]
\item \label{item:1countablygenerated} $\mathrm{Fin}$; 
\item \label{item:2countablygenerated} $\mathrm{Fin}\times \emptyset:=\{S\subseteq \omega^2: \exists n \in \omega, S\subseteq [0,n]\times \omega\}$;
\item \label{item:3countablygenerated} $\mathrm{Fin}\oplus \mathcal{P}(\omega):=\{S\subseteq \{0,1\}\times \omega: |S\cap (\{0\}\times \omega)|<\infty\}$. 
\end{enumerate} 
(Recall that two ideals $\mathcal{I}_1$ and $\mathcal{I}_2$ on countable sets $H_1$ and $H_2$, respectively, are called \emph{isomorphic}, written as $\mathcal{I}_1 \simeq \mathcal{I}_2$ if there exists a bijection $h: H_1\to H_2$ such that $h[S] \in \mathcal{I}_2$ if and only if $S \in \mathcal{I}_1$ for all $S\subseteq H_1$; accordingly, it is easy to see that $\mathcal{I}_{\mathrm{P}} \simeq \mathrm{Fin}\times \emptyset$, and that the ideals in \ref{item:1countablygenerated}-\ref{item:3countablygenerated} are pairwise nonisomorphic.)   
This has been essentially proved in \cite[Proposition 1.2.8]{MR1711328}, however the correct statement appears 
in \cite[Section 2]{MR3543775}. 

We include its simple proof for the sake of completeness. Suppose that $\mathcal{I}$ is countably generated by a partition $\{Q_j: j \in\omega\}$ of $\omega$, and define $J:=\{j \in \omega: Q_j \in \mathrm{Fin}\}$. 
If $J$ is empty then every $Q_j$ is infinite, hence $\mathcal{I}\simeq \mathrm{Fin}\times \emptyset$. If $J$ is nonempty finite then $Q_j$ is infinite for infinitely many $j\in \omega$, hence $\mathcal{I}\simeq \mathcal{J}$, where
$$
\mathcal{J}:=\mathrm{Fin}\oplus (\mathrm{Fin}\times \emptyset):=\{S\subseteq \omega \cup \omega^2: S\cap \omega \in \mathrm{Fin}, S\cap \omega^2 \in \mathrm{Fin}\times \emptyset\}. 
$$
However, $\mathcal{J}\simeq \mathrm{Fin}\times \emptyset$, with the witnessing bijection $h: \omega\cup \omega^2\to \omega^2$ defined by $h(a,b)=(a,b+1)$ and $h(a)=(a,0)$ for all $a,b \in \omega$. Hence, let us assume hereafter that $J$ is infinite. If $J^c$ is empty then $\mathcal{I}\simeq \mathrm{Fin}$. If $J^c$ is nonempty finite then $\mathcal{I}\simeq \mathrm{Fin}\oplus \mathcal{P}(\omega)$. Lastly, if also $J^c$ is infinite, then $\mathcal{I}\simeq \mathcal{J}\simeq \mathrm{Fin}\times \emptyset$. 
\end{rmk}


\subsection{Unbounded to Bounded case}

In this section, we may assume that $\mathcal{I}\neq \mathrm{Fin}$, otherwise we go back in the previous cases. 
Differently from the other results, we are going to show that, quite often, there are no matrices $A$ 
which satisfy
\begin{equation}\label{eq:boundedunboundeddefinition}
A \in (c(X,\mathcal{I}),c^b(Y,\mathcal{J})) 
\,\,\, \text{ and }\,\,\,
\mathcal{J}\text{-}\lim A\bm{x}=T(\,\mathcal{I}\text{-}\lim \bm{x})
\,\text{ for all }\bm{x} \in c(X,\mathcal{I}),
\end{equation}
unless $T$ is the zero operator. To this aim, recall that an ideal $\mathcal{I}$ is said to be \emph{tall} if, for every infinite set $S\subseteq \omega$, there exists an infinite subset $S^\prime \subseteq S$ which belongs to $\mathcal{I}$ (note that countably generated ideals, hence also $\mathrm{Fin}$, are not tall; for a necessary condition in the case of countably generated ideals $\mathcal{J}$ and arbitrary $\mathcal{I}$, cf. Remark \ref{rmk:necessryT5sharp} below). 

\begin{thm}\label{main:IJREGULARUnboundedBounded}
Fix a nonzero linear operator $T \in \mathcal{L}(X,Y)$, where $X,Y$ are Banach spaces. 
Let also $\mathcal{I},\mathcal{J}$ be an ideals on $\omega$ such that $\mathcal{I}$ is tall. 

Then there are no matrices $A=(A_{n,k})$ of linear operators in $\mathcal{L}(X,Y)$ which satisfy \eqref{eq:boundedunboundeddefinition}. 
\end{thm}

Of course, if $T=0$, then the zero matrix $A$ (namely, the matrix with $A_{n,k}=0$ for all $n,k$) satisfies \eqref{eq:boundedunboundeddefinition}. 
However, this is essentially the unique possibility: 
\begin{thm}\label{thm:cIc_0bJ}
Let $X,Y$ be Banach spaces. Let also $\mathcal{I}, \mathcal{J}$ be ideals on $\omega$ such that $\mathcal{I}$ is tall. 

Then a matrix $A=(A_{n,k})$ of linear operators in $\mathcal{L}(X,Y)$ belongs to $(c(X,\mathcal{I}), c_0^b(Y,\mathcal{J}))$ if and only if there exists $k_1 \in \omega$ such that\textup{:}
\begin{enumerate}[label={\rm (\textsc{B}\arabic{*})}]
\item \label{item:B1} $A_{n,k}=0$ for all $n \in \omega$ and $k\ge k_1$\textup{;}
\item \label{item:B2} $\sup_n \|A_{n,k}x\|<\infty$ for all $x \in X$ and $k<k_1$\textup{;}
\item \label{item:B3} $\mathcal{J}\text{-}\lim_n A_{n,k}=0$ for all $k<k_1$\textup{.}
\end{enumerate}
\end{thm}

As it will turn out, condition \ref{item:B1} is satisfied also for all matrices in the larger class $(c_{00}(X,\mathcal{I}), \ell_\infty(Y))$, provided that $\mathcal{I}$ is tall. 

%


\subsection{Unbounded to Unbounded case} 
In this last section, we study the analogue condition for a matrix $A=(A_{n,k})$ of linear operators in $\mathcal{L}(X,Y)$ to satisfy 
\begin{equation}\label{eq:unoundedunboundeddefinition}
A \in (c(X,\mathcal{I}),c(Y,\mathcal{J})) 
\,\,\, \text{ and }\,\,\,
\mathcal{J}\text{-}\lim A\bm{x}=T(\,\mathcal{I}\text{-}\lim \bm{x})
\,\text{ for all }\bm{x} \in c(X,\mathcal{I}).
\end{equation}
\begin{rmk}\label{rmk:existence}
In some cases, it is easy to provide examples of matrices which satisfy \eqref{eq:unoundedunboundeddefinition}. 
Indeed, suppose that $T\in \mathcal{B}(X,Y)$ and $\mathcal{I}\subseteq \mathcal{J}$. We claim that the matrix $A=(A_{n,k})$ such that $A_{n,k}=T$ if $n=k$ and $A_{n,k}=0$ otherwise has this property. 
For, set $\bm{x} \in c(X,\mathcal{I})$ with $\mathcal{I}$-limit $\eta$. 
Then 
$
\mathcal{I}\text{-}\lim\nolimits_n A_n\bm{x}=\mathcal{I}\text{-}\lim\nolimits_n Tx_n=T(\mathcal{I}\text{-}\lim \bm{x})=T\eta,
$ 
which implies $\mathcal{J}\text{-}\lim A\bm{x}=T\eta$. Here, we used the fact the $T$ preserves $\mathcal{I}$-convergence: this is clear if $T=0$, otherwise $\{n \in\omega: \|Tx_n-T\eta\|<\varepsilon\}\supseteq \{n \in \omega: \|x_n-\eta\|<\varepsilon/\|T\|\} \in \mathcal{I}^\star$ for all $\varepsilon>0$.
\end{rmk}

In the next results, we will need a further weakening of \ref{item:T1flat} and stronger versions of conditions \ref{item:T3} and \ref{item:T5}, that is, 
\begin{enumerate}[label={\rm (\textsc{T}\arabic{*}$^{\flat\flat}$)}]
\setcounter{enumi}{0}
\item \label{item:T1flatflat} For all $n \in \omega$, there exists $f(n) \in \omega$ such that $\|A_{n,\ge f(n)}\|<\infty$\textup{;}
\end{enumerate} 
\begin{enumerate}[label={\rm (\textsc{T}\arabic{*}$^{\sharp}$)}]
\setcounter{enumi}{2}
\item \label{item:T3sharp} $\sum\nolimits_kA_{n,k}x_k$ converges in the norm of $Y$ for all $\bm{x} \in c(X,\mathcal{I})$ and $n \in \omega$\textup{;}
\end{enumerate}
\begin{enumerate}[label={\rm (\textsc{T}\arabic{*}$^{\sharp}$)}]
\setcounter{enumi}{4}
\item \label{item:T5sharp} $A \in (c_{00}(X,\mathcal{I}),c_0(Y,\mathcal{J}))$\textup{.}
\end{enumerate}

\begin{thm}\label{main:IJREGULARUnboundedUnbounded}
Fix a linear operator $T \in \mathcal{L}(X,Y)$, where $X,Y$ are Banach spaces. 
Let also $\mathcal{I},\mathcal{J}$ be an ideals on $\omega$. 
%
Then a matrix $A=(A_{n,k})$ of linear operators in $\mathcal{L}(X,Y)$ satisfies \eqref{eq:unoundedunboundeddefinition} 
if 
there exists $k_0 \in \omega$ such that \ref{item:T1flat}, \ref{item:T3sharp}, \ref{item:T4}, and \ref{item:T5sharp} hold. 

Conversely, if $A$ satisfies \eqref{eq:unoundedunboundeddefinition} then 
\ref{item:T1flatflat}, \ref{item:T3sharp}, \ref{item:T4}, and \ref{item:T5sharp} hold. 
\end{thm}

It turns out that we obtain a complete characterization if $\mathcal{J}$ is a strongly selective ideal:
\begin{thm}\label{thm:cXcYrapid}
Fix a linear operator $T \in \mathcal{L}(X,Y)$, where $X,Y$ are Banach spaces. 
Let also $\mathcal{I},\mathcal{J}$ be an ideals on $\omega$ such that $\mathcal{J}$ is a strongly selective ideal.  

Then a matrix $A=(A_{n,k})$ of linear operators in $\mathcal{L}(X,Y)$ satisfies \eqref{eq:unoundedunboundeddefinition} 
if and only if 
there exists $k_0 \in \omega$ such that \ref{item:T1flat}, \ref{item:T3sharp}, \ref{item:T4}, and \ref{item:T5sharp} hold. 

In addition, if each $A_{n,k}$ is bounded, it is possible to choose $k_0=f(n)=0$ for all 
$n \in \omega\setminus J_0$. 
\end{thm}

Some additional properties can be obtained in special cases:
\begin{rmk}
Suppose that $\mathcal{I}$ is tall ideal. Then \ref{item:T3sharp} implies, thanks to Lemma \ref{lem:dualc00I}, that $A$ is \emph{row finite}, namely, $\{k \in \omega: A_{n,k}\neq 0\} \in \mathrm{Fin}$ for all $n\in \omega$. 
\end{rmk}
\begin{rmk}\label{rmk:necessryT5sharp}
Suppose that $\mathcal{J}$ is a countably generated ideal. Then \ref{item:T5sharp} implies, thanks to Theorem  \ref{thm:Xomegaellinfty}, that 
for every infinite $E \in \mathcal{I}$ there exists $J \in \mathcal{J}^\star$ such that $\{k \in E: A_{n,k}\neq 0 \text{ for some }n \in J\}$ is finite. 
\end{rmk}


\section{Preliminaries}\label{sec:preliminaries}


Unless otherwise stated, 
we assume that $X,Y$ are Banach spaces. 
We recall the following results on the so-called K\"othe--Toeplitz $\beta$-duals: 
\begin{lem}\label{lem:convergenceoperator}
Let $(T_k)$ be a sequence of linear operators in $\mathcal{L}(X,Y)$. Then $\sum_k T_kx_k$ is convergent in the norm of $Y$ for all sequences $\bm{x} \in \ell_\infty(X)$ if and only if\textup{:} 
\begin{enumerate}[label={\rm (\textsc{N}\arabic{*})}]
\item \label{item:N2} $\|T_{\ge k_0}\|<\infty$ for some $k_0 \in \omega$\textup{;}
\item \label{item:N3} $\lim_k\|T_{\ge k}\|=0$\textup{.}
\end{enumerate}
In addition, if each $T_k$ is bounded, it is possible to choose $k_0=0$. 
\end{lem}
\begin{proof}
See \cite[Proposition 3.1 and Proposition 3.3]{MR568707}. 
\end{proof}
In particular, \ref{item:N2} implies that $T_k$ is bounded for all $k\ge k_0$. 
We remark that, if $X=\mathbf{R}$ and each linear operator $T_k$ can be written as $T_kx=xy_k$, for some $y_k \in Y$, then the sequence $(T_k)$ of Lemma \ref{lem:convergenceoperator} is also called \textquotedblleft bounded multiplier convergent,\textquotedblright\, see e.g. \cite{MR704294}. 


\begin{lem}\label{lem:convergenceoperatorbetadualc(X)}
Let $(T_k)$ be a sequence of linear operators in $\mathcal{L}(X,Y)$. Then $\sum_k T_kx_k$ is norm convergent in $Y$ for all $\bm{x} \in c(X)$ if and only if 
\ref{item:N2} holds for some $k_0$, together with\textup{:}
\begin{enumerate}[label={\rm (\textsc{N}\arabic{*}$^\prime$)}]
\setcounter{enumi}{1}
\item \label{item:N2prime} $\sum_kT_k$ converges in the strong operator topology\textup{.}
\end{enumerate}
In addition, if each $T_k$ is bounded, it is possible to choose $k_0=0$. 
\end{lem}
\begin{proof}
See \cite[Proposition 3.2]{MR568707}. 
\end{proof}

A characterization of the K\"othe--Toeplitz $\beta$-dual of a sequence space which is strictly related to $c(\mathcal{Z})$ can be found in \cite[Theorem 4]{MR261212}. 

However, if $X$ is finite dimensional, we have a simpler characterization:
\begin{cor}\label{cor:finitedimensionalVSextremepoints}
Let $(T_k)$ be a sequence of linear operators in $\mathcal{L}(X,Y)$ and assume, in addition, that $X$ is finite dimensional. 
Then the following are equivalent:
\begin{enumerate}[label={\rm (\roman*)}]
\item \label{item:1corollaryfinitedim} $\sum_kT_kx_k$ is norm convergent in $Y$ for all sequences $\bm{x} \in \ell_\infty(X)$\textup{;}
\item \label{item:2corollaryfinitedim} $\sum_kT_kx_k$ is norm convergent in $Y$ for all sequences $\bm{x} \in c(X)$\textup{;}
\item \label{item:3corollaryfinitedim} $\|T_\omega\|<\infty$\textup{.}
\end{enumerate}
\end{cor}
\begin{proof}
The implication \ref{item:1corollaryfinitedim} $\implies$ \ref{item:2corollaryfinitedim} is clear and \ref{item:2corollaryfinitedim} $\implies$ \ref{item:3corollaryfinitedim} follows by Lemma \ref{lem:convergenceoperatorbetadualc(X)}. Indeed, since $d:=\mathrm{dim}(X)<\infty$, each $T_{k}$ is bounded. 

\ref{item:3corollaryfinitedim} $\implies$ \ref{item:1corollaryfinitedim} It follows by Lemma \ref{lem:convergenceoperator} that it is enough to prove that condition \ref{item:N3} holds, provided that $\|T_\omega\|<\infty$.  
To this aim, suppose for the sake of contradiction that there exists $\varepsilon>0$ such that $\limsup_k \|T_{\ge k}\|>\varepsilon$. 
Then there exist a sequence $\bm{x}$ taking values in the closed unit ball $B_X$ and a partition $\{I_j: j \in \omega\}$ of $\omega$ in consecutive finite intervals such that 
$$
\forall j \in \omega, \quad 
\|(T_k: k \in I_j)\|\ge \left\|\sum\nolimits_{k \in I_j}T_kx_k\right\|>\varepsilon.
$$
Assume without loss of generality that $X=\mathbf{R}^d$ and, since every norm is equivalent, endow it with the $1$-norm $\|x\|:=\sum_i|x_i|$. 
Define the sequence $\bm{y}$ by $y_j:=\sum\nolimits_{k \in I_j}T_kx_k$ for all $j$. 
Let $\{Q_1,\ldots,Q_{2^d}\}$ be the collection of all closed quadrants of $\mathbf{R}^d$. 
Since $\bigcup_{i\le 2^d}Q_i=\mathbf{R}^d$, there exist $i_0 \in [1,2^d]$ and an infinite set $J\subseteq \omega$ such that $y_j \in Q_{i_0}$ for all $j \in J$. 
It follows that
\begin{displaymath}
\begin{split}
\|T_\omega\|\ge \|(T_k: k \in J)\| &
\ge \sup_{F \subseteq J, F \in \mathrm{Fin}}\left\|\sum\nolimits_{j \in F}y_j\right\| \\
&=\sup_{F \subseteq J, F \in \mathrm{Fin}} \sum\nolimits_{j \in F}\left\|y_j\right\| 
\ge \sup_{F \subseteq J, F \in \mathrm{Fin}}|F|\varepsilon =\infty, 
\end{split}
\end{displaymath}
which contradicts the standing hypothesis. 
\end{proof}

\begin{lem}\label{lem:finitedimensionalnormkjfdhgd}
Let $(T_k)$ be a sequence of linear operators in $\mathcal{L}(X,Y)$ and assume, in addition, that $X$ is finite dimensional and $\mathcal{J}\text{-}\lim_k \|T_kx\|=0$ for all $x \in X$, where $\mathcal{J}$ is an ideal on $\omega$. Then $\mathcal{J}\text{-}\lim_k \|T_k\|=0$. 
\end{lem}
\begin{proof}
Since $d:=\mathrm{dim}(X)<\infty$, each $T_k$ is bounded. 
Assume without loss of generality that $X=\mathbf{R}^d$ and endow it with the $1$-norm as in the proof of Corollary \ref{cor:finitedimensionalVSextremepoints}. 
Note that the set $\mathscr{E}$ of extreme points of the closed unit ball is finite. 
Hence, for each $k \in \omega$, there exists $e_k \in \mathscr{E}$ such that $\|T_k\|=\|T_ke_k\|$. 
It follows that
$$
\mathcal{J}\text{-}\lim\nolimits_k \|T_k\|\le 
\sum\nolimits_{e \in \mathscr{E}}\mathcal{J}\text{-}\lim\nolimits_k \|T_ke\|=0,
$$
which completes the proof. 
\end{proof}

As anticipated, we will need an ideal version of the Banach--Steinhaus theorem, which has been recently obtained in \cite{MR4172859}. 
\begin{thm}\label{thm:uniformJboundedness}
Let $\mathcal{J}$ be a rapid$^+$ $P^+$-ideal on $\omega$. Also, let $(T_n)$ be a sequence of linear operators in $B(X,Y)$ and suppose that 
$$
\forall x\in X, \quad 
\mathcal{J}\text{-}\limsup\nolimits_n \|T_nx\|<\infty.
$$
Then $\mathcal{J}\text{-}\limsup_n\|T_n\|<\infty$. 
\end{thm}
\begin{proof}
It follows by \cite[Theorem 3.1(b)]{MR4172859}.
\end{proof}


The following result on unbounded operators is due to Lorentz and Macphail \cite{MR52533} in the case $\mathcal{J}=\mathrm{Fin}$, see also \cite[Theorem 4.1]{MR568707} for a textbook exposition. 
\begin{thm}\label{thm:lorentzmacphail}
Let $(T_n)$ be a sequence of linear operators in $\mathcal{L}(X,Y)$. Let also $(M_n)$ be a decreasing sequence of closed linear subspaces of $X$ such that each $T_n$ is bounded on $M_n$. 

Lastly, fix a strongly selective ideal $\mathcal{J}$ on $\omega$ and suppose that $(T_nx) \in \ell_\infty(Y,\mathcal{J})$ for all $x \in X$. There there exist $n_0 \in \omega$ and $J^\star \in \mathcal{J}^\star$ such that $T_n$ is bounded on $M_{n_0}$ for all $n \in J^\star$.
\end{thm}
\begin{proof}
For each $n \in \omega$, define 
$$
S_n:=\left\{k \in \omega: T_k \text{ is not bounded on }M_n\right\}.
$$ 
Note that $(S_n)$ is a decreasing sequence and $S_n \cap [0,n]=\emptyset$ for all $n \in\omega$. First, suppose that there exists $n_0 \in \omega$ such that $S_{n_0} \in \mathcal{J}$. It follows that $T_k$ is bounded on $M_{n_0}$ for all $k \in J^\star:=\omega \setminus S_{n_0} \in \mathcal{J}^\star$. Hence, suppose hereafter that $(S_n)$ is a decreasing sequence in $\mathcal{J}^+$. Since $\mathcal{J}$ is strongly selective, there exists $S \in \mathcal{J}^+$ with increasing enumeration $(k_n)$ with the property that $k_{n+1} \in S_{k_n}$ for all $n \in \omega$. In other words, we have
$$
\{k_n: n \in \omega\} \in \mathcal{J}^+ 
\quad \text{ and }\quad 
T_{k_{n+1}} \text{ is not bounded on }M_{k_n} \text{ for all }n \in \omega.
$$ 
It follows by the proof in \cite[Theorem 4.1]{MR568707} that there exists $x \in X$ such that $\|T_{k_{n+1}}x\|\ge n$ for all $n\ge 2$. This contradicts the hypothesis that the sequence $(T_nx)$ is $\mathcal{J}$-bounded. 
\end{proof}

\begin{thm}\label{thm:ctoell(X)}
Let $A=(A_{n,k})$ be a matrix of linear operators in $\mathcal{L}(X,Y)$. 
Let $\mathcal{J}$ be a strongly selective ideal on $\omega$. 
Then $A \in (c(X), \ell_\infty(Y, \mathcal{J}))$ only if 
there exists $k_0 \in \omega$ which satisfies \ref{item:T1flat}. 
%

In addition, if each $A_{n,k}$ is bounded, it is possible to choose $k_0=f(n)=0$ for all 
$n \in \omega\setminus J_0$. 
\end{thm}
\begin{proof}
Thanks to 
Lemma \ref{lem:convergenceoperatorbetadualc(X)}, for each $n \in\omega$, there exists an integer $f(n) \in \omega$ such that $\|A_{n,\ge f(n)}\|<\infty$. 
Without loss of generality, we can suppose that the sequence $(f(n): n \in \omega)$ is weakly increasing. 
Now, for each $n \in \omega$, define 
$$
M_n:=\{\bm{x} \in c(X): x_k=0 \text{ for all }k< f(n)\}.
$$
Then $(M_n)$ is a decreasing sequence of closed linear subspaces of the Banach space $c(X)$. 
In addition, for each $n \in \omega$, the linear operator $A_n: c(X)\to Y$ is well defined. 
Thanks to the Banach--Steinhaus theorem, $A_n$ is bounded on $M_n$ for all $n \in \omega$.  

At this point, it follows by Theorem \ref{thm:lorentzmacphail} that there exist $n_0\in \omega$ and $J^\star \in \mathcal{J}^\star$ such that $A_n$ is bounded on $M_{n_0}$ for all $n \in J^\star$. 
Thanks to Theorem \ref{thm:uniformJboundedness}, we obtain that $\mathcal{J}\text{-}\limsup_n \|A_n\upharpoonright M_{n_0}\|<\infty$, i.e., there exist a constant $\kappa >0$ and $J_0 \in \mathcal{J}^\star$, with $J_0=\omega$ if $\mathcal{J}=\mathrm{Fin}$, such that $\|A_n\bm{x}\|\le \kappa \|\bm{x}\|$ for all $\bm{x} \in M_{n_0}$ and $n \in J_0$. 
To complete the proof, observe that 
$$
\forall n \in J_0, 
\forall \bm{x} \in c_{00}(X) \cap M_{n_0}, \quad 
\|A_n\bm{x}\|=\left\|\sum\nolimits_{k\ge f(n_0)}A_{n,k}x_k\right\| \le \kappa \|\bm{x}\|, 
$$
which implies that $\|A_{n,\ge f(n_0)}\|\le \kappa$. 
Since the upper bound is independent of $n\in J_0$, the claim follows by setting $k_0:=f(n_0)$. 

The second part is clear once we observe that it is possible to choose $f(n)=0$ for all $n$.  
\end{proof}


In the following results we will need the following weakening of \ref{item:T2}, namely, 
\begin{enumerate}[label={\rm (\textsc{T}\arabic{*}$^\flat$)}]
\setcounter{enumi}{1}
\item \label{item:T2flat} $\mathcal{J}\text{-}\limsup_n\|A_{n,k}x\|<\infty$ for all $x \in X$ and $k<k_0$\textup{.}
\end{enumerate}
It is clear that \ref{item:T5} implies \ref{item:T2flat}, which is the reason why it does not appear in Theorem \ref{main:IJREGULARboundedUnbounded}, cf. Remark \ref{rmk:conditionT5}.


\begin{thm}\label{thm:maddoxmain}
Let $A=(A_{n,k})$ be a matrix of linear operators in $\mathcal{L}(X,Y)$. 
Also, let $\mathcal{J}$ be a strongly selective ideal on $\omega$. 
Then $A \in  (\ell_\infty(X), \ell_\infty(Y, \mathcal{J}))$ 
if and only if there exists $k_0 \in \omega$ such that 
\ref{item:T1flat}, \ref{item:T2flat}, and \ref{item:T3natural} hold. 

In addition, if each $A_{n,k}$ is bounded, it is possible to choose $k_0=0$. 
\end{thm}
\begin{proof}
\textsc{If part.} Fix $\bm{x}\in \ell_\infty(X)$. 
Thanks to Lemma \ref{lem:convergenceoperator}, for each $n$, the sum $A_n\bm{x}=\sum_kA_{n,k}x_k$ is convergent in the norm of $Y$, hence $A\bm{x}$ is well defined. 
It follows by \ref{item:T2flat} that there exist $J_1 \in \mathcal{J}^\star$ and $\kappa>0$ such that $\|A_{n,k}x_k\|\le \kappa$ for all $k<k_0$ and $n \in J_1$. 
Hence
\begin{equation}\label{eq:boundedness}
\left\|A_n\bm{x}\right\| 
\le \sum\nolimits_{k<k_0}\left\|A_{n,k}x_k\right\|+\left\|\sum\nolimits_{k\ge k_0}A_{n,k}x_k\right\|
\le \kappa k_0+\|\bm{x}\|\sup\nolimits_{t\in J_0} \|A_{t, \ge k_0}\|
\end{equation}
for all $n \in J_0\cap J_1 \in \mathcal{J}^\star$, which proves that $A\bm{x} \in \ell_\infty(Y, \mathcal{J})$. 

\medskip

\textsc{Only If part.} 
The necessity of \ref{item:T1flat} follows by Theorem \ref{thm:ctoell(X)}, with $k_0=0$ if each $A_{n,k}$ is bounded. 
Now, if \ref{item:T2flat} does not hold, there would exist $x \in X$ and $k<k_0$ such that $\mathcal{J}\text{-}\limsup_n\|A_{n,k}x\|=\infty$. This contradicts the hypothesis that  $A \in  (\ell_\infty(X), \ell_\infty(Y,\mathcal{J}))$ by choosing $\bm{x} \in \ell_\infty(X)$ such that $x_t=x$ if $t=k$ and $x_t=0$ otherwise. 
Lastly, the necessity of \ref{item:T3natural}  follows by Lemma \ref{lem:convergenceoperator}.  
\end{proof}

\begin{rmk}\label{rmk:conditionL1tooweak}
The same example given in Remark \ref{rmk:regulardoesnotimplylinftylinfty} proves that, even if $X=Y$, $\mathcal{J}=\mathrm{Fin}$, and each $A_{n,k}$ is bounded (so that \ref{item:T2flat} is void), condition \ref{item:T1flat} (as its stronger version \ref{item:T1}) is not sufficient to characterize the matrix class $(\ell_\infty(X), \ell_\infty(Y,\mathcal{J}))$; cf. Corollary \ref{cor:finitedimensionallinfty} below for the finite dimensional case. 
\end{rmk}

\begin{thm}\label{thm:maddoxmaincXellinfty}
Let $A=(A_{n,k})$ be a matrix of linear operators in $\mathcal{L}(X,Y)$. 
Also, let $\mathcal{J}$ be a strongly selective ideal on $\omega$. 
Then $A \in  (c(X), \ell_\infty(Y,\mathcal{J}))$ if and only if there exists $k_0 \in \omega$ such that \ref{item:T1flat} and \ref{item:T2flat} hold, together with\textup{:}
\begin{enumerate}[label={\rm (\textsc{T}\arabic{*}$^\flat$)}]
\setcounter{enumi}{3}
\item \label{item:T4flat} $\sum_{k}A_{n,k}$ convergences in the strong operator topology for all $n$\textup{.}
\end{enumerate}

In addition, if each $A_{n,k}$ is bounded, it is possible to choose $k_0=0$. 
\end{thm}
\begin{proof}
The proof goes as in Theorem \ref{thm:maddoxmain}, replacing Lemma \ref{lem:convergenceoperator} with Lemma \ref{lem:convergenceoperatorbetadualc(X)}.
\end{proof}

However, if $X$ is finite dimensional, it is possible to simplify the equivalences in Theorem \ref{thm:maddoxmain} and Theorem \ref{thm:maddoxmaincXellinfty}, namely, condition \ref{item:T1flat} is necessary and sufficient in both cases choosing $k_0=0$. 
\begin{cor}\label{cor:finitedimensionallinfty}
Let $A=(A_{n,k})$ be a matrix of linear operators in $\mathcal{L}(X,Y)$ and assume, in addition, that $X$ is finite dimensional. 
Also, let $\mathcal{J}$ be a strongly selective ideal on $\omega$. 

Then the following are equivalent\textup{:}
\begin{enumerate}[label={\rm (\roman*)}]
\item \label{item:1corollaryfinitedimellinfty} $A \in (\ell_\infty(X), \ell_\infty(Y,\mathcal{J}))$\textup{;}
\item \label{item:2corollaryfinitedimellinfty} $A \in (c(X), \ell_\infty(Y,\mathcal{J}))$\textup{;}
\item \label{item:3corollaryfinitedimellinfty} \ref{item:T1flat} holds with $k_0=0$\textup{.}
\end{enumerate}
\end{cor}
\begin{proof}
\ref{item:1corollaryfinitedimellinfty} $\implies$ \ref{item:2corollaryfinitedimellinfty} is clear, 
\ref{item:2corollaryfinitedimellinfty} $\implies$ \ref{item:3corollaryfinitedimellinfty} follows by Theorem \ref{thm:ctoell(X)}, and 
\ref{item:3corollaryfinitedimellinfty} $\implies$ \ref{item:1corollaryfinitedimellinfty} follows by Lemma \ref{lem:convergenceoperator}, Corollary \ref{cor:finitedimensionalVSextremepoints}, Theorem \ref{thm:maddoxmain}, and the fact that each $A_{n,k}$ is bounded. 
\end{proof}

\begin{lem}\label{lemma:unboundedcase}
Let $(T_k)$ be a sequence of nonzero linear operators in $\mathcal{L}(X,Y)$. Then there exists a sequence $\bm{x} \in X^\omega$ such that $\sum_kT_kx_k$ is not convergent in the norm of $Y$. 
\end{lem}
\begin{proof}
For each $k\in \omega$, pick $y_k \in X$ such that $T_ky_k\neq 0$. Now define the sequence $\bm{x}$ recursively as it follows: $x_0:=y_0$ and, if $x_0,\ldots,x_{n-1}$ are defined for some $n\ge 1$, then $x_n:=\kappa_n y_n$, where $\kappa_n:=(n+\|\sum_{k\le n-1}T_kx_k\|)/\|T_ny_n\|$. Indeed, it follows that 
$$
\left\|\sum\nolimits_{k\le n}T_kx_k\right\|\ge \|T_nx_n\|-\left\|\sum\nolimits_{k\le n-1}T_kx_k\right\|=n
$$
for all $n\ge 1$, completing the proof.  
\end{proof}

\begin{lem}\label{lem:Itall}
Let $\mathcal{I}$ be an ideal on $\omega$ which is tall. Then there exists a partition $\{E_n: n \in \omega\}$ of $\omega$ such that $E_n \in \mathcal{I}\cap \mathrm{Fin}^+$ for all $n$. 
\end{lem}
\begin{proof}
We define such partition recursively, with the property that $\{0,\ldots,n\}\cup E_0\cup \cdots E_n$ for all $n$. 
Since $\mathcal{I}$ is tall there exists an infinite set $S_0 \in \mathcal{I}$, and set $E_0:=S_0\cup \{0\}$. Now, suppose that $E_0,\ldots,E_{n-1}$ have been defined, for some $n\ge 1$, so that the set $G_n:=\omega\setminus \bigcup_{i<n}E_i$ belongs to $\mathcal{I}^\star$. Since $\mathcal{I}$ is tall there exists an infinite set $S_n \in \mathcal{I}$ contained in $G_n$. The claim follows defining $E_n:=S_n$ if $n \notin G_n$ and $E_n:=S_n\cup \{n\}$ otherwise. 
\end{proof}

\begin{lem}\label{lem:dualc00I}
Let $(T_k)$ be a sequence of linear operators in $\mathcal{L}(X,Y)$. Let also $\mathcal{I}$ be a tall ideal on $\omega$.  Then $\sum_kT_kx_k$ is convergent in the norm of $Y$ for all $\bm{x} \in c_{00}(X,\mathcal{I})$ if and only if 
$\{k \in \omega: T_k \neq 0\}$ is finite. 
\end{lem}
\begin{proof}
Thanks to Lemma \ref{lem:Itall}, there exists a partition $\{E_n: n \in \omega\}$ of $\omega$ such that $E_n \in \mathcal{I}\cap \mathrm{Fin}^+$ for all $n$. 
Now, it follows by Lemma \ref{lemma:unboundedcase} that $F_n:=\{k \in E_n: T_k\neq 0\} \in \mathrm{Fin}$. 
Define $F:=\bigcup_n F_n$ and suppose for the sake of contradiction that $F\notin \mathrm{Fin}$. 
Since $\mathcal{I}$ is tall, there exists an infinite subset $F^\prime \subseteq F$ such that $F^\prime \in \mathcal{I}$. 
However, by construction $\{k \in F^\prime: T_k\neq 0\}=F^\prime$, which contradicts Lemma \ref{lemma:unboundedcase}. 
\end{proof}

Recalling that a matrix $A$ is said to be row-finite if $\{k \in \omega: A_{n,k}\neq 0\} \in \mathrm{Fin}$ for all $n\in \omega$,  
we provide strong necessary conditions on the entries of matrices in $(X^\omega, \ell_\infty(Y,\mathcal{J}))$ and $(c_{00}(X,\mathcal{I}), \ell_\infty(Y))$, where $\mathcal{J}$ is countably generated and $\mathcal{I}$ is tall. 
\begin{thm}\label{thm:Xomegaellinfty}
Let $A=(A_{n,k})$ be a matrix of linear operators in $\mathcal{L}(X,Y)$ such that $A \in (X^\omega, \ell_\infty(Y,\mathcal{J}))$, where $\mathcal{J}$ is a countably generated ideal on $\omega$. 
Then $A$ is row finite and there exist $J \in \mathcal{J}^\star$ and $k_1 \in \omega$ such that $A_{n,k}=0$ for all $n \in J$ and $k\ge k_1$. 
\end{thm}
\begin{proof}
First, suppose that $A \in (X^\omega, \ell_\infty(Y,\mathcal{J}))$ and that $\mathcal{J}$ is generated by a sequence of increasing sets $(Q_j)$. 
Then $A$ is row finite by Lemma \ref{lemma:unboundedcase}, so that $F_n:=\{k \in\omega: A_{n,k}\neq 0\} \in \mathrm{Fin}$ for all $n \in \omega$. 
Suppose for the sake of contradiction that 
\begin{equation}\label{eq:contradictionXomegaellinfty}
\forall J \in \mathcal{J}^\star, \forall k_1 \in \omega, \exists k\ge k_1, \exists n \in J, \quad A_{n,k}\neq 0.
\end{equation}
Define a strictly increasing sequence $(s_n)$ in $\omega$ such that $s_0:=\min\{n \in \omega: F_n \neq \emptyset\}$, and, recursively, $s_{n+1}:=\min(S_n\setminus Q_{j_n})$ for all $n\in \omega$, where
$$
S_n:=\{t \in \omega: \max F_t>\max(F_{s_0}\cup\cdots\cup F_{s_n})\}
\,\, \text{ and }\,\,
j_n:=\min\{j \in \omega: s_n \in Q_j\}.
$$
We claim that $S_n\setminus Q_{j_n}$ is nonempty, so that $s_{n+1}$ is well defined. Note that $S_n \in \mathcal{J}^+$: indeed, in the opposite, we would contradict \eqref{eq:contradictionXomegaellinfty} by setting $J=S_n^c$ and $k_1=1+\max(F_{s_0}\cup\cdots\cup F_{s_n})$. 
In addition, since $Q_{j_n} \in \mathcal{J}$, it follows that $S_n\setminus Q_{j_n}\in \mathcal{J}^+$ (in particular, it is nonempty). 

Since the set $S:=\{s_n: n \in \omega\}$ is not contained in any $Q_j$, it follows that $S \in \mathcal{J}^+$. In addition, set $k_{n}:=\max F_{s_n}$ for all $n \in \omega$. Using a technique similar to the one used in Lemma \ref{lemma:unboundedcase}, we are going to construct a sequence $\bm{x}$ supported on $K:=\{k_n:n \in \omega\}$ such that $\|A_n\bm{x}\|\ge n$ for all $n\in \omega$, so that $A\bm{x}\notin \ell_\infty(Y,\mathcal{J})$. 
To this aim, pick an arbitrary sequence $\bm{y} \in X^\omega$ such that $A_{s_n,k_n}y_{n}\neq 0$ for all $n \in \omega$. 
Now define $\bm{x}$ recursively so that $x_{k_0}:=y_0$ and, if $x_{k_0},\ldots,x_{k_{n-1}}$ are given for some $n\ge 1$, then 
$x_{k_n}:=\kappa_n y_n$,  
where $\kappa_n:=(n+\|\sum_{i\le n-1}A_{s_n,k_i}x_{k_i}\|)/\|A_{s_n,k_n}y_n\|$. 
Indeed, it follows that 
\begin{displaymath}
\begin{split}
\|A_{s_n}\bm{x}\|
=\left\|\sum\nolimits_{k \in F_{s_n}}A_{s_n,k}x_{k}\right\|
&=\left\|\sum\nolimits_{i\le n}A_{s_n,k_i}x_{k_i}\right\|\\
&\ge \kappa_n\|A_{s_n,k_n}y_{k_n}\|-\left\|\sum\nolimits_{i\le n-1}A_{s_n,k_i}x_{k_i}\right\|=n
\end{split}
\end{displaymath}
for all $n\ge 1$. Since $S\in \mathcal{J}^+$, we conclude that $\mathcal{J}\text{-}\limsup_n\|A_n\bm{x}\|\neq 0$. 
\end{proof}

\begin{thm}\label{thm:c00ellinfty}
Let $A=(A_{n,k})$ be a matrix of linear operators in $\mathcal{L}(X,Y)$ such that $A \in (c_{00}(X,\mathcal{I}), \ell_\infty(Y))$, where $\mathcal{I}$ is a tall ideal on $\omega$. 
Then there exists $k_1 \in \omega$ such that $A_{n,k}=0$ for all $n \in \omega$ and $k\ge k_1$. 
\end{thm}
\begin{proof}
First, $A$ is row finite by Lemma \ref{lem:dualc00I}. Now, it is enough to repeat the proof of Lemma \ref{lem:dualc00I} replacing Lemma \ref{lemma:unboundedcase} with Theorem \ref{thm:Xomegaellinfty}, with $F_n:=\{k \in \omega: \exists m \in \omega, A_{m,k}\neq 0\}$. 
\end{proof}

\section{Key tools}\label{sec:technical}

Let $A=(A_{n,k})$ be a matrix of linear operators in $\mathcal{L}(X,Y)$ and note that, for each $n \in \omega$ and sequence $\bm{x}$ taking values on the closed unit ball $B_X$, we have 
\begin{equation}\label{eq:inequalitytoverifyequality}
\mathcal{J}\text{-}\limsup\nolimits_n\|A_{n,\omega}\| \ge \mathcal{J}\text{-}\limsup\nolimits_n\|A_n\bm{x}\|
\end{equation}

In this section, we provide sufficient conditions on the matrix $A$ and on the ideal $\mathcal{J}$ for the existence of a sequence $\bm{x}$ such that the above inequality is actually an equality. 
The following result is the operator version of the sliding jump argument contained in the proof of \cite[Theorem 1.3]{ConnorLeo} for the one-dimensional case $X=Y=\mathbf{R}$ and $\mathcal{J}=\mathrm{Fin}$.  

\begin{thm}\label{thm:key}
Let $\mathcal{J}$ be an ideal on $\omega$ which is countably generated. 
Let also $A=(A_{n,k})$ be a matrix of linear operators in $\mathcal{L}(X,Y)$ which satisfies conditions \ref{item:T1flat} for some $k_0 \in \omega$, \ref{item:T3natural}, and 
\ref{item:T6flat}. 
Then there exists a sequence $\bm{x}$ with values on the unit sphere $S_X$ such that
$$
\mathcal{J}\text{-}\limsup\nolimits_{n}\|A_{n,\omega}\|=
\mathcal{J}\text{-}\limsup\nolimits_{n}\|A_n\bm{x}\|.
$$
%
\end{thm}
\begin{proof}
Suppose that $\mathcal{J}$ is generated by an increasing sequence of sets $(Q_n)$, and define 
$$
\eta_0:=\mathcal{J}\text{-}\limsup\nolimits_{n}\|A_{n,\omega}\| \in [0,\infty].
$$
It follows by \ref{item:T6flat} that 
\begin{equation}\label{eq:c0normfinitecolumns}
\forall E \in \mathrm{Fin}, \quad 
\mathcal{J}\text{-}\lim\nolimits_n\|A_{n, E}\|=\sum\nolimits_{k \in E}\mathcal{J}\text{-}\lim\nolimits_n\|A_{n, k}\|=0.
\end{equation}
Hence there exists $J_1 \in \mathcal{J}^\star$ such that $\|A_{n, \le k_0}\|\le 1$ for all $n\in J_1$. 
Setting $J_2:=J_0\cap J_1 \in \mathcal{J}^\star$, it follows that 
$$
\forall n\in J_2, \quad 
\|A_{n,\omega}\| \le \|A_{n, \le k_0}\|+\|A_{n,\ge k_0}\| \le 1+\sup\nolimits_{t\in J_0}\|A_{t,\ge k_0}\|,
$$
which is finite by condition \ref{item:T1flat}. 
This proves that $\eta_0 \neq \infty$. 
Moreover, if $\eta_0=0$, it is enough to let $\bm{x}$ be an arbitrary sequence with values in $S_X$, thanks to \eqref{eq:inequalitytoverifyequality}. 
Hence, let us assume hereafter that $\eta_0 \in (0,\infty)$. 

At this point, for each $n \in \omega$, define the set 
\begin{equation}\label{eq:SdefinitioninJplus}
E_n:=\left\{t \in J_2: \left|\|A_{t,\omega}\|-\eta_0\right|\le \frac{\eta_0}{2^{n}}\right\},
\end{equation}
and note that $(E_n)$ is a decreasing sequence of sets in $\mathcal{J}^+$. 
Define also two strictly increasing sequences $(s_n)$ and $(m_n)$ of nonnegative integers, a descreasing sequence $(H_n)$ of sets of $\mathcal{J}^\star$, and a decreasing sequence $(S_n)$ of sets in $\mathcal{J}^+$ as it follows. Set $s_0:=\min S_0$, with $S_0:=E_0$, $H_0:=\omega$, and choose $m_0 \in \omega$ such that $\|A_{s_0, \ge m_0}\|\le \eta_0$, which is possible by \ref{item:T3natural}. 
Now, suppose that $s_0,\ldots,s_{n-1},m_0,\ldots,m_{n-1} \in \omega$ and the sets $S_0,\ldots,S_{n-1}\in \mathcal{J}^+$ and $H_0,\ldots,H_{n-1} \in \mathcal{J}^\star$ have been defined for some $n\ge 1$. 
\begin{enumerate}[label=(\roman*)]
\item \label{item:condition_setHn} Set $H_n:=H_{n-1} \cap \left\{t \in \omega: \|A_{t,m_{n-1}}\|\le \frac{\eta_0}{2^n}\right\}$, so that $H_n \in \mathcal{J}^\star$ by  \eqref{eq:c0normfinitecolumns}; 
\item \label{item:condition_setSn} Define $S_n:=E_n \cap H_n$, 
hence $S_n \in \mathcal{J}^+$ and
\begin{equation}\label{eq:conditionsn}
\forall t\in S_n, \quad \|A_{t,\le m_{n-1}}\|\le \frac{\eta_0}{2^{n}}.
\end{equation}
\item \label{item:condition_sn}  
Choose $s_n \in S_{n}\setminus Q_z$, where $k$ is an integer such that $s_{n-1} \in Q_k$. In particular, $s_n>s_{n-1}$. Note that this is possible since $S_{n}\setminus Q_k \in \mathcal{J}^+$. 
\item \label{item:condition_mn} Lastly, thanks to \ref{item:T3natural}, choose $m_n>m_{n-1}$ such that 
\begin{equation}\label{eq:condition3}
\|A_{s_n, \ge m_n}\| \le \frac{\eta_0}{2^{n}}.
\end{equation}
\end{enumerate}

To conclude the proof, let $\bm{x}=(x_n)$ be a sequence taking values on the unit sphere $S_X$ such that 
\begin{equation}\label{eq:condition4}
\forall n \ge 1, \quad 
\left\|\sum\nolimits_{k \in M_n}A_{s_n,k}x_{k}\right\|\ge \left\|A_{s_n,M_n}\right\| -\frac{\eta_0}{2^{n}},
\end{equation}
where $M_n:=\{k \in \omega: m_{n-1}<k\le m_n\}$, and $x_n$ is arbitrarily chosen on the unit sphere $S_X$ for all $n \in [0,m_0]$. 
%
It follows by \eqref{eq:SdefinitioninJplus}, \eqref{eq:conditionsn}, \eqref{eq:condition3}, and \eqref{eq:condition4} that  
\begin{equation}\label{eq:chaindifficult}
\begin{split}
\forall n\ge 1, \quad 
\|A_{s_n}\bm{x}\| &
\ge \left\|\sum\nolimits_{k \in M_n}A_{s_n,k}x_{k}\right\|-\left\|A_{s_n,\le m_{n-1}}\right\|-\left\|A_{s_n,\ge m_{n}}\right\|\\
&\ge \left\|A_{s_n,M_n}\right\|-\left\|A_{s_n,\le m_{n-1}}\right\|-\left\|A_{s_n,\ge m_{n}}\right\|-\frac{\eta_0}{2^{n}} \\
&\ge \left\|A_{s_n,\omega}\right\|-2\left\|A_{s_n,\le m_{n-1}}\right\|-2\left\|A_{s_n,\ge m_{n}}\right\|-\frac{\eta_0}{2^{n}}\\
&\ge \eta_0\left(1-\frac{1}{2^{n}}\right)-\frac{\eta_0}{2^{n-1}}-\frac{\eta_0}{2^{n-1}}-\frac{\eta_0}{2^{n}} >\eta_0\left(1-\frac{1}{2^{n-3}}\right).\\
\end{split}
\end{equation}
It follows by construction that we cannot find an integer $k$ such that $s_n \in Q_k$ for all $n \in \omega$. Therefore $\{s_n: n \in \omega\} \in \mathcal{J}^+$ and $\mathcal{J}\text{-}\limsup_n \|A_n\bm{x}\| \ge \eta_0$. 

The converse inequality follows by \eqref{eq:inequalitytoverifyequality}, completing the proof. 
\end{proof}

Related results (in the case $X=Y=\mathbf{R}$ and $\mathcal{J}=\mathrm{Fin}$) can be found in \cite[Lemma 3.1]{MR27351} and \cite[Corollary 12]{MR241957}.  
The following corollary is immediate (we omit details):
\begin{cor}\label{cor:nontrivialJFin}
With the same hypothesis of Theorem \ref{thm:key}, for each  $E\subseteq \omega$, there exists a sequence $\bm{x}$ 
taking values on $S_X\cup \{0\}$ and supported on $E$ such that 
$$
\mathcal{J}\text{-}\limsup\nolimits_{n}\|A_{n,E}\|=
\mathcal{J}\text{-}\limsup\nolimits_{n}\|A_n\bm{x}\|.
$$ 
\end{cor}

In the finite dimensional case, the statement can be simplified: 
\begin{cor}\label{cor:finitedimensionalkey}
With the same notations of Corollary \ref{cor:finitedimensionmain}, suppose that $X=\mathbf{R}^d$, $Y=\mathbf{R}^m$, and that $\mathcal{J}$ is countably generated. Let also $A$ be a matrix which satisfies\textup{:}
\begin{enumerate}[label={\rm (\textsc{K}\arabic{*})}]
\item \label{item:K1} $\sup_{n \in J_0} \sum_k\sum_{i,j}\left|a_{n,k}(i,j)\right|<\infty$ for some $J_0 \in \mathcal{J}^\star$\textup{;}
\item \label{item:K2} $\sum_k\sum_{i,j}\left|a_{n,k}(i,j)\right|<\infty$ for all $n \in \omega$\textup{;}
\item \label{item:K3} $\mathcal{J}\text{-}\lim_n \sum_{i,j}\left|a_{n,k}(i,j)\right|=0$ for all $k \in \omega$\textup{.}
\end{enumerate}
Then, for each $E\subseteq \omega$, there exists a sequence $\bm{x}=(x^{(0)},x^{(1)},\ldots)$ 
taking values on $S_X\cup \{0\}$ and supported on $E$ such that 
$$
\mathcal{J}\text{-}\limsup\nolimits_{n}\sum\nolimits_{k \in E}\max\nolimits_j\sum\nolimits_{i}\left|a_{n,k}(i,j)\right|=
\mathcal{J}\text{-}\limsup\nolimits_{n}
\sum\nolimits_i \left|\sum\nolimits_k \sum\nolimits_j a_{n,k}(i,j) x^{(k)}_j \right|.
$$ 
\end{cor}
\begin{proof}
Conditions \ref{item:K1} and \ref{item:K2} correspond to \ref{item:T1flat}, and \ref{item:K3} corresponds to \ref{item:T6flat}, cf. the proof of Corollary \ref{cor:finitedimensionmain}. 
In addition, \ref{item:T1flat} implies \ref{item:T3natural} by Proposition \ref{prop:implications}.\ref{item:5simplification} and Remark \ref{rmk:ont1flat} below. 
The claim follows by Corollary \ref{cor:nontrivialJFin}. 
\end{proof}

\begin{rmk}\label{rmk:conditionT5insteadofT6flat}
As it has been previously observed in Remark \ref{rmk:evilgroupnorm}, there exists a matrix $A$ which does not satisfy condition \ref{item:T6flat} with $\mathcal{J}=\mathrm{Fin}$ (which corresponds to \ref{item:S3sharp}) and, on other hand, it satisfies \ref{item:T5} with $\mathcal{I}=\mathcal{J}=\mathrm{Fin}$ (which corresponds to \ref{item:S3}). 
In addition, it is immediate to see that $A$ satisfies  \ref{item:T1} and \ref{item:T3natural} since $A_{n,k}=0$ for all $n\ge 0$ and $k>0$. 
However, in this case, the conclusion of Theorem \ref{thm:key} fails: indeed, if $\bm{x}=(x_0,x_1,\ldots)$ is a sequence taking values on the closed unit ball of $\ell_2$ then $\lim_n\|A_n\bm{x}\|=\lim_n\|A_{n,0}x_0\|= 0$, and $\|A_{n,\omega}\|=\|A_{n,0}\|=1$ for all $n\in \omega$. 
Hence we cannot replace \ref{item:T6flat} in Theorem \ref{thm:key}
with the weaker pointwise condition $\mathcal{J}\text{-}\lim_nA_{n,k}=0$ for all $k\in \omega$ (namely, \ref{item:T5} with $\mathcal{I}=\mathrm{Fin}$). 
\end{rmk}

However, the weaker condition above is sufficient to obtain the same claim if $\eta_0=\infty$; 
note that, as it has been shown above, this case is impossible with the hypotheses of Theorem \ref{thm:key}.
\begin{thm}\label{thm:keyfakeinfty}
Let $\mathcal{J}$ be an ideal on $\omega$ which is countably generated. Let also $A=(A_{n,k})$ be a matrix of linear operators in $\mathcal{L}(X,Y)$ which satisfies conditions \ref{item:T3natural} and \ref{item:T5} with $\mathcal{I}=\mathrm{Fin}$. In addition, assume that $\mathcal{J}\text{-}\limsup\nolimits_n \|A_{n,\ge f(n)}\|=\infty$. 

Then there exists a sequence $\bm{x}$ with values on $S_X$ such that $\mathcal{J}\text{-}\limsup\nolimits_n\|A_n\bm{x}\|=\infty$. 
\end{thm}
\begin{proof}
We proceed with the same strategy of the proof of Theorem \ref{thm:key}. 
Note that \ref{item:T3natural} implies \ref{item:T1flatflat}. 
Accordingly, define $E_n:=\{t \in \omega: \|A_{t,\ge f(t)}\| \ge n\}$, which belongs to $\mathcal{J}^+$ for each $n \in \omega$. 
Set $s_0:=\min S_0$, with $S_0:=E_0$, $H_0:=\omega$, choose $m_0 \in \omega$ such that $\|A_{s_0,\ge m_0}\|\le 1$, which is possible by \ref{item:T3natural}, and pick some arbitrary vectors $x_0,\ldots,x_{m_0} \in S_X$. 

Now, suppose that, for some $n\ge 1$, all the integers $s_i, m_i$, sets $S_i, H_i\subseteq \omega$, and vectors $x_j$ have been defined for all $i\le n-1$ and $j\le m_{n-1}$. 
Then, define recursively 
$$
H_n:=H_{n-1}\cap \left\{t \in\omega: \left\|\sum\nolimits_{k\le m_{n-1}}A_{t,k}x_k\right\|\le 1\right\},
$$
which belongs to $\mathcal{J}^\star$ thanks to \ref{item:T5} with $\mathcal{I}=\mathrm{Fin}$. Define $S_n$ and $s_n$ as in the proof of Theorem \ref{thm:key}, and $m_n>m_{n-1}$ such that $\|A_{s_n,\ge m_n}\|\le 1$, which is possible again by \ref{item:T3natural}. Finally, we choose some vectors $\{x_k: k \in M_n\}$ on the unit sphere $S_X$ such that $\|\sum_{k \in M_n}A_{s_n,k}x_k\| \ge \|A_{s_n,M_n}\|-1$. 
(Here, differently from the previous proof, the sequence $\bm{x}$ has been constructed recursively.) 
Reasoning as in \eqref{eq:chaindifficult}, we conclude that 
$$
\forall n\ge 1, \quad \|A_{s_n}\bm{x}\|\ge n-5 
\quad \text{ and }\quad 
\{s_t: t \in \omega\} \in \mathcal{J}^+.
$$
Therefore $\mathcal{J}\text{-}\limsup_n \|A_n\bm{x}\|=\infty$. 
\end{proof}

As a consequence of the results above, we obtain an ideal version of the Hahn--Schur theorem (where the classical version corresponds to the case $\mathcal{J}=\mathrm{Fin}$): 
\begin{thm}\label{thm:hahnschur}
Let $\mathcal{J}$ be a countably generated ideal. Let also $A=(a_{n,k})$ be an infinite real matrix such that 
$\sum_k|a_{n,k}|<\infty$ for all $n \in \omega$ and 
\begin{equation}\label{eq:hahnschurhypothesis}
\forall E \subseteq \omega, \quad 
\mathcal{J}\text{-}\lim\nolimits_n \sum\nolimits_{k \in E}a_{n,k}=0.
\end{equation}
Then $\mathcal{J}\text{-}\lim\nolimits_n \sum\nolimits_{k}|a_{n,k}|=0$. 
\end{thm}
\begin{proof}
Set $\eta_0:=\mathcal{J}\text{-}\limsup\nolimits_n \sum_k|a_{n,k}| \in [0,\infty]$. Note also that the standing hypotheses imply \ref{item:T3natural} and \ref{item:T6flat}. In addition, if $\eta_0<\infty$, then \ref{item:T1} holds. 
It follows by Theorem \ref{thm:key} and Theorem \ref{thm:keyfakeinfty} that there exists a real sequence $\bm{x}$ taking values in $\{1,-1\}$ such that 
\begin{equation}\label{eq:hahnschur}
\mathcal{J}\text{-}\limsup\nolimits_n \sum\nolimits_k a_{n,k}x_k
=\eta_0.
\end{equation}
At this point, define $y_n:=(1+x_n)/2$ for all $n \in\omega$, so that $\bm{y}=\bm{1}_E$, where $E:=\{k \in \omega: x_k=1\}$. It follows by \eqref{eq:hahnschurhypothesis} and \eqref{eq:hahnschur} that
\begin{displaymath}
\begin{split}
\eta_0 &=\mathcal{J}\text{-}\limsup\nolimits_n \left(\sum\nolimits_k a_{n,k}(2y_k-1)\right)\\
&=2\cdot \mathcal{J}\text{-}\limsup\nolimits_n \sum\nolimits_k a_{n,k}y_k - \mathcal{J}\text{-}\lim\nolimits_n \sum\nolimits_k a_{n,k}\\
&=2\cdot \mathcal{J}\text{-}\limsup\nolimits_n \sum\nolimits_{k\in E} a_{n,k} =0.
\end{split}
\end{displaymath}
Therefore $\mathcal{J}\text{-}\lim\nolimits_n \sum\nolimits_{k}|a_{n,k}|=0$.
\end{proof}


Neither \ref{item:T6flat} nor \ref{item:T5} with $\mathcal{I}=\mathrm{Fin}$ 
will be required for a conclusion as in Theorem \ref{thm:key} in the case of positive linear operators between certain Banach lattices. 
\begin{prop}\label{prop:Anonnegativelema}
Let $X$ be an AM-space 
with order unit $e>0$ and let $Y$ be a Banach lattice. 
Also, let $T=(T_k)$ be a sequence of linear operators in $\mathcal{L}(X,Y)$ which are positive and such that $\sum_k T_kx_k$ is convergent in the norm of $Y$ for all sequences $\bm{x} \in \ell_\infty(X)$. 
Then 
%
$$
\forall E\subseteq \omega, \quad 
\|(T_k: k \in E)\|=\left\|\sum\nolimits_{k\in E}T_ke\right\|.
$$
\end{prop}
\begin{proof}
Thanks to \cite[Theorem 3.40]{MR2011364}, 
there exists a (unique, up to homeomorphism) compact Hausdorff space $K$ and a lattice isometry $h: X\to C(K)$ such that $h(e)$ is the constant function $\bm{1}$, where $C(K)$ is the Banach lattice of continuous functions $f:K\to \mathbf{R}$, endowed with the supremum norm. 
It follows that the closed unit ball $B_X$ of $X$ is simply the order interval $[-e,e]$, indeed 
$$
B_X=h^{-1}(\{f \in C(K): \|f\|\le 1\})=h^{-1}([-\bm{1},\bm{1}])=[-e,e].
$$ 

At this point, let us fix a nonempty $E\subseteq \omega$, 
a positive integer $n$ with $n\le |E|$, 
distinct integers $i_1,\ldots,i_n \in E$, 
and vectors $x_1,\ldots,x_n \in X$ such that $|x_k|\le e$ for all $k \in \{1,\ldots,n\}$. 
By the fact that each $T_{k}$ is a positive linear operator and 
\cite[Theorem 1.7(2)]{MR2011364}, it follows that 
$$
0\le \left|\sum\nolimits_{k\le n}T_{i_k}x_k\right|
\le  \sum\nolimits_{k\le n}\left|T_{i_k}x_k\right|
\le  \sum\nolimits_{k\le n}T_{i_k}\left|x_k\right|
\le  \sum\nolimits_{k\le n}T_{i_k}e.
$$
By the definition of group norm and the compatibility between the norm and the order structure in $X$, we obtain 
$$
\|(T_k: k \in E)\|=\sup\nolimits_{n}\left\|\sum\nolimits_{k\le n, k \in E}T_ke\right\|
=\left\|\sum\nolimits_{k\in E}T_ke\right\|, 
$$
which concludes the proof. 
\end{proof}

In the same spirit of Corollary \ref{cor:nontrivialJFin}, we obtain the following consequence: 
\begin{cor}\label{cor:Anonnegativemain}
Let $X$ be an AM-space with order unit $e>0$ and let $Y$ be a Banach lattice. 
Also, let $\mathcal{J}$ be an ideal of $\omega$, and $A=(A_{n,k})$ be a matrix of positive linear operators such that $A \in (\ell_\infty(X), \ell_\infty(Y))$. 
Then
$$
\forall E\subseteq \omega, \quad 
\mathcal{J}\text{-}\limsup\nolimits_n \|A_{n,E}\|=\mathcal{J}\text{-}\limsup\nolimits_n 
\|A_n\bm{x}\|,
$$
where $x_n=e$ if $n \in E$ and $x_n=0$ otherwise. 
\end{cor}
\begin{proof}
Proposition \ref{prop:Anonnegativelema} implies that $\|A_{n,E}\|=\|\sum\nolimits_{k \in E}A_{n,k}e\|=\|A_n\bm{x}\|$ for all $n$. 
\end{proof}

\section{Main proofs}\label{sec:mainproofs}

\begin{proof}[Proof of Theorem \ref{main:IJREGULAR}]
\textsc{If part.} Suppose that \ref{item:T1}-\ref{item:T5} hold and fix a bounded sequence $\bm{x}$ which is $\mathcal{I}$-convergent to $\eta \in X$. 
Thanks to \ref{item:T3}, $A_n\bm{x}$ is well defined for each $n \in \omega$. In addition, we obtain by \ref{item:T1} and \ref{item:T2} that Inequality \eqref{eq:boundedness} holds for all $n$ (indeed, in our case $J_0=J_1=\omega$), proving 
that $A\bm{x} \in \ell_\infty(Y)$. 

First, suppose that $\eta=0$. 
Fix $\varepsilon>0$ and define 
$$
\delta:=\frac{\varepsilon}{1+\sup_n\|A_{n,\ge k_0}\|} 
$$
and $E:=\{n \in \omega: n<k_0 \text{ or }\|x_n\|>\delta\}$. 
Note that $E \in \mathcal{I}$ and that, again by \ref{item:T3}, $\left(\sum_{k \in E}A_{n,k}x_k: n \in \omega\right)$ is a well-defined sequence in $Y$. 
In addition, 
$$
S:=\left\{n \in\omega: \left\|\sum\nolimits_{k \in E}A_{n,k}x_k\right\|>\delta\right\} \in \mathcal{J}
$$ 
by \ref{item:T5}. 
Now, suppose that $\|A_n\bm{x}\|> \varepsilon$ for some $n \in \omega$. It follows that 
\begin{equation}\label{eq:ddjyhkjfgf}
\varepsilon< 
\left\|\sum\nolimits_{k \in E}A_{n,k}x_k\right\|+ \left\|\sum\nolimits_{k \notin E}A_{n,k}x_k\right\| \le 
\left\|\sum\nolimits_{k \in E}A_{n,k}x_k\right\|+ \delta \sup\nolimits_t\|A_{t,\ge k_0}\|.
\end{equation}
By the definition of $\delta$, this implies that 
$$
\left\|\sum\nolimits_{k \in E}A_{n,k}x_k\right\| >\varepsilon-\delta\sup\nolimits_t\|A_{t,\ge k_0}\|= \delta, 
$$
so that $n \in S$. We conclude that $\{n \in \omega: \|A_n\bm{x}\|>\varepsilon\}\subseteq S \in \mathcal{J}$. By the arbitrariness of $\varepsilon$, we obtain $\mathcal{J}\text{-}\lim A\bm{x}=0$. 

At this point, suppose that $\eta \in X$ and define $\bm{y} \in X^\omega$ such that $y_n:=x_n-\eta$ for all $n$. Note that $\bm{y} \in c_0^b(X,\mathcal{I})$, hence by the previous case $A\bm{y}\in c_{0}^b(Y,\mathcal{J})$. It follows by \ref{item:T4} that 
$$
\mathcal{J}\text{-}\lim A\bm{x}=\mathcal{J}\text{-}\lim A\bm{y}+\mathcal{J}\text{-}\lim\nolimits_n\sum\nolimits_kA_{n,k}\eta=T\eta,
$$
which proves that $A$ is $(\mathcal{I}, \mathcal{J})$-regular with respect to $T$. 

\medskip

\textsc{Only If part.} Assume that $A$ is $(\mathcal{I}, \mathcal{J})$-regular with respect to $T$. Hence, the matrix $A$ belongs, in particular, to $(c(X), \ell_\infty(Y))$. 
It follows by Theorem \ref{thm:maddoxmaincXellinfty} that conditions \ref{item:T1} and \ref{item:T2} hold. 
Also, for each $n \in \omega$, the sum $\sum_kA_{n,k}x_k$ is convergent in the norm of $Y$ for all sequences $\bm{x} \in c^b(X,\mathcal{I})$, hence \ref{item:T3} holds. 
Moreover, for each $x \in X$, the constant sequence $(x,x,\ldots)$ has $\mathcal{I}$-limit $x$, hence $\mathcal{J}\text{-}\lim_n\sum\nolimits_kA_{n,k}x=Tx$, which is condition \ref{item:T4}. 

Lastly, fix $\bm{x}\in c_{00}^b(X,\mathcal{I})$, so that $\mathcal{I}\text{-}\lim \bm{x}=0$. By the $(\mathcal{I}, \mathcal{J})$-regularity of $A$ with respect to $T$, we obtain that $A\bm{x}$ is well defined and $\mathcal{J}\text{-}\lim A\bm{x}=T(0)=0$. This proves \ref{item:T5}.

\medskip 

The second part of the statement follows by Proposition \ref{prop:implications}.\ref{item:4simplification}. 
\end{proof}


\begin{proof}
[Proof of Proposition \ref{prop:implications}] 
\ref{item:1simplification} Condition \ref{item:T4} implies that $\sum_kA_{n,k}$ is convergent in strong operator topology for each $n \in \omega$. The conclusion follows by Lemma \ref{lem:convergenceoperatorbetadualc(X)}. 

\ref{item:2simplification} For all $k \in \omega$ and $x \in X$, if $\lim_n A_{n,k}x=0$ then $\sup_n \|A_{n,k}x\|<\infty$.

\ref{item:3simplification} It follows by Lemma \ref{lem:convergenceoperator}. 

\ref{item:4simplification} Since $A$ belongs to $(c(X), \ell_\infty(Y))$, the claim follows by Theorem \ref{thm:ctoell(X)}. 

\ref{item:5simplification} It is known that each $A_{n,k}$ is bounded. The second part follows by Lemma \ref{lem:convergenceoperator} and Corollary \ref{cor:finitedimensionalVSextremepoints}. (Note that this is not necessarily true if $\mathrm{dim}(X)=\infty$, cf. Remark \ref{rmk:conditionL1tooweak}.)

\ref{item:6simplification} Given $k \in \omega$ and $x \in X$, the sequence $\bm{x}$ defined by $x_n=x$ if $n=k$ and $x_n=0$ otherwise has $\mathcal{I}$-limit $0$, hence $\mathcal{J}\text{-}\lim_nA_{n,k}x=T(0)=0$. This claim follows by Lemma \ref{lem:finitedimensionalnormkjfdhgd}.

\ref{item:7simplification} For all $n,k \in \omega$, there exists $a_{n,k} \in \mathbf{R}$ such that $A_{n,k}=a_{n,k}A_0$. Note also that
\begin{equation}\label{eq:groupnormlinearcombination}
\forall n \in \omega, \forall E\subseteq \omega, \quad 
\|A_{n,E}\|=\|A_0\|\sum\nolimits_{k \in E}|a_{n,k}|.
\end{equation}
Hence condition \ref{item:T1} implies that $A_0$ is necessarily bounded, $k_0=0$ can be chosen by point \ref{item:4simplification} above, and hence $\sup_n\sum\nolimits_k|a_{n,k}|<\infty$. Thanks to \eqref{eq:groupnormlinearcombination}, it is immediate to conclude that $\lim_k\|A_{n,\ge k}\|=\|A_0\|\lim_k \sum_{t\ge k}|a_{n,t}|=0$ for all $n \in \omega$.

\ref{item:8simplification} By point \ref{item:7simplification} above, $A_0$ is bounded. Moreover, if $A_0=0$ the claim is obvious. Otherwise, there exists $x \in X$ such that $A_0x\neq 0$ and by \ref{item:T5} we obtain 
$\mathcal{J}\text{-}\lim_n A_{n,k}x=0$, so that $\|A_0x\|\cdot \mathcal{J}\text{-}\lim_n|a_{n,k}|=0$ for all $k \in \omega$. It follows that $\mathcal{J}\text{-}\lim_n\|A_{n,k}\|=\|A_0\|\cdot \mathcal{J}\text{-}\lim_n|a_{n,k}|=0$. 
\end{proof}

\begin{rmk}\label{rmk:ont1flat}
As it evident from the above proofs, the statements of Proposition \ref{prop:implications}.\ref{item:1simplification}, 
\ref{item:3simplification}, 
\ref{item:5simplification}, 
\ref{item:7simplification}, and 
\ref{item:8simplification} 
are correct also replacing \ref{item:T1} with the weaker condition \ref{item:T1flatflat}. 
\end{rmk}


\begin{proof}
[Proof of Theorem \ref{thm:JfinIJregularmain}]
First, assume that $A$ is $(\mathcal{I}, \mathcal{J})$-regular with respect to $T$, with $\mathcal{J}$ countably generated.  
Then \ref{item:T1} and \ref{item:T4} holds by Theorem  \ref{main:IJREGULAR}. 
Now, fix $\bm{x} \in c_{00}^b(X, \mathcal{I})$. Then $\mathcal{I}\text{-}\lim \bm{x}=0$ and, since $A$ is $(\mathcal{I}, \mathcal{J})$-regular, we obtain $\mathcal{J}\text{-}\lim A\bm{x}=0$, namely, $\mathcal{J}\text{-}\lim_n \|A_n\bm{x}\|=0$. Then \ref{item:T6} holds by Corollary \ref{cor:nontrivialJFin}.

Conversely, assume that \ref{item:T1}, \ref{item:T4}, and \ref{item:T6} hold. Then \ref{item:T5} holds (since it is implied by \ref{item:T6}), and conditions \ref{item:T2} and \ref{item:T3} hold by Proposition \ref{prop:implications}.\ref{item:2simplification} and \ref{item:3simplification}, respectively. It follows by Theorem \ref{main:IJREGULAR} that $A$ is $(\mathcal{I}, \mathcal{J})$-regular with respect to $T$. 
\end{proof}


\begin{proof}
[Proof of Theorem \ref{thm:AMspaceregularmain}]
The first proof goes verbatim as in the proof of Theorem \ref{thm:JfinIJregularmain}, replacing Corollary \ref{cor:nontrivialJFin} with Corollary \ref{cor:Anonnegativemain}. 
Also the second part proceeds similarly, with the difference that \ref{item:T2} holds by Proposition \ref{prop:implications}.\ref{item:4simplification}: indeed, thanks to \cite[Theorem 4.3]{MR2262133}, each positive linear operator $A_{n,k}$ between Banach lattices is necessarily continuous. 
\end{proof}


\begin{proof}
[Proof of Corollary \ref{cor:finitedimensionmain}]
Endow $\mathbf{R}^d$ and $\mathbf{R}^m$ with the corresponding $1$-norm, as in the proof of Corollary \ref{cor:finitedimensionalVSextremepoints}, so that 
$\|A_{n,k}\|=\max\nolimits_j\sum\nolimits_i|a_{n,k}(i,j)|$, 
cf. e.g. \cite[Example 5.6.4]{MR2978290}. 
Also, note that, for all $n \in \omega$ and $E\subseteq \omega$, 
\begin{equation}\label{eq:upperlowerboundsnormmatrix}
\frac{1}{d}\sum\nolimits_{k \in E}\sum\nolimits_{i,j}|a_{n,k}(i,j)|\le 
\|A_{n,E}\| \le 
\sum\nolimits_{k \in E}\sum\nolimits_{i,j}|a_{n,k}(i,j)|.
\end{equation}
Accordingly, conditions \ref{item:F1}, \ref{item:F4}, and \ref{item:F6} are simply a rewriting of \ref{item:T1}, \ref{item:T4}, and \ref{item:T6}, respectively, with $k_0=0$ (which can be chosen thanks to Proposition \ref{prop:implications}.\ref{item:5simplification} and \ref{item:4simplification}; in particular, in the following subcases, condition \ref{item:T2} is void). 

\medskip

\textsc{Case $\mathcal{I}=\mathrm{Fin}$:} First, assume that $A$ is $(\mathcal{I},\mathcal{J})$-regular with respect to $T$. 
Then \ref{item:T1} and \ref{item:T4} hold by Theorem  \ref{main:IJREGULAR}. 
In addition, by Proposition \ref{prop:implications}.\ref{item:6simplification} also \ref{item:T6flat} holds. 
At this point, fix $E \in \mathrm{Fin}$ and note that $\|A_{n,E}\|\le \sum_{k \in E}\|A_{n,k}\|$. 
Taking $\mathcal{J}$-limits on both sides, we obtain \ref{item:T6}. 
Conversely, assume that \ref{item:T1}, \ref{item:T4} and \ref{item:T6} hold. Then \ref{item:T6} implies \ref{item:T5}. And \ref{item:T3} holds by Proposition \ref{prop:implications}.\ref{item:3simplification} and \ref{item:5simplification}. Hence $A$ is $(\mathcal{I}, \mathcal{J})$-regular with respect to $T$ by Theorem \ref{main:IJREGULAR}. 

\medskip

\textsc{Case $\mathcal{J}$ is countably generated:} If $A$ is $(\mathcal{I},\mathcal{J})$-regular with respect to $T$, then \ref{item:T1} and \ref{item:T4} hold by Theorem  \ref{main:IJREGULAR}. In addition, \ref{item:T3natural} and \ref{item:T6flat} hold by Proposition \ref{prop:implications}.\ref{item:5simplification} and \ref{item:6simplification}. Then \ref{item:T6} holds by Theorem \ref{thm:JfinIJregularmain}. The converse goes as in the previous case. 

\medskip

\textsc{Case $a_{n,k}(i,j)\ge 0$ for all $1\le i\le m$, $1\le j\le d$, and $n,k \in \omega$:} Note that $\mathbf{R}^d$ is an AM-space with order unit $(1,\ldots,1)$. The proof goes on the same lines of the previous case replacing Theorem \ref{thm:JfinIJregularmain} with Theorem \ref{thm:AMspaceregularmain}.
\end{proof}

\begin{proof}
[Proof of Corollary \ref{cor:llinftyc0}]
Let $\mathcal{I}$ and $\mathcal{I}^\prime$ be two maximal ideals such that $\{2\omega, 2\omega+1\}\subseteq \mathcal{I}\cup \mathcal{I}^\prime$. 
Since $c^b(\mathbf{R}^d,\mathcal{I})=c^b(\mathbf{R}^d,\mathcal{I}^\prime)=\ell_\infty(\mathbf{R}^d)$, we obtain that \ref{item:F6} holds with both $E=2\omega$ and $E=2\omega+1$, hence it is equivalent to \ref{item:F6prime}. 
In turn, \ref{item:F6prime} implies \ref{item:F4} with $T=0$. 
The claim follows by Corollary \ref{cor:finitedimensionmain}.
\end{proof}

\begin{proof}
[Proof of Corollary \ref{cor:eachAnkmultiple}] 
First, suppose that $A$ is $(\mathcal{I},\mathcal{J})$-regular with respect to $T$. 
It follows by Theorem \ref{main:IJREGULAR} that conditions \ref{item:T1}-\ref{item:T5} hold. 
Thanks to Equation \eqref{eq:groupnormlinearcombination} above and \ref{item:T1}, there exists $k_0 \in \omega$ such that $\sup_{n}\|A_{n,\ge k_0}\|=\|A_0\|\sup_n\sum_{k\ge k_0}|a_{n,k}|<\infty$ (since $A_0\neq 0$), which implies \ref{item:M0} and \ref{item:M1}. Condition \ref{item:M4} is just a rewriting of \ref{item:T4}. In addition, conditions \ref{item:T3natural} and \ref{item:T6flat} hold by Proposition \ref{prop:implications}.\ref{item:7simplification}-\ref{item:8simplification}. 
It follows that \ref{item:M6}, which is just a rewriting of \ref{item:T6}, holds for the same reasons in the proof of Corollary \ref{cor:finitedimensionmain}. 

Conversely, assume that conditions \ref{item:M0}-\ref{item:M6} hold. \ref{item:M0} implies that each $A_{n,k}$ is bounded, hence it is possible to choose $k_0=0$ by Proposition \ref{prop:implications}.\ref{item:4simplification}, hence \ref{item:T2} holds. 
Accordingly, \ref{item:M1}, \ref{item:M4}, and \ref{item:M6} are just  rewritings of \ref{item:T1}, \ref{item:T4}, and \ref{item:T6}, respectively. 
Lastly, \ref{item:T3} follows by Proposition \ref{prop:implications}.\ref{item:3simplification} and \ref{item:7simplification}, and \ref{item:T5} is implied by \ref{item:T6}. 
To sum up, conditions \ref{item:T1}-\ref{item:T5} hold, and the conclusion follows by Theorem \ref{main:IJREGULAR}. 
\end{proof}

\begin{proof}
[Proof of Theorem \ref{main:IJREGULARboundedUnbounded}] 
It proceeds on the same lines of the proof of Theorem \ref{main:IJREGULAR} (recalling the \ref{item:T5} implies \ref{item:T2flat}), with the only difference that it is not necessarily true that $J_0=J_1=\omega$ but only $J_0,J_1 \in \mathcal{J}^\star$. 
\end{proof}

\begin{proof}
[Proof of Corollary \ref{cor:RHgeneral}] 
Proceeding as in the proof of Corollary \ref{cor:finitedimensionmain}, note that conditions \ref{item:R1} and \ref{item:R2} are simply a rewriting of \ref{item:T1flat}, \ref{item:R4} is a rewriting of \ref{item:T4}, and \ref{item:R6} is a rewriting of \ref{item:T6}. Since each $A_{n,k}$ is bounded, we choose $k_0=f(n)=0$ for all $n \in Q_{t_0}$ in the statement of Theorem  \ref{main:IJREGULARboundedUnbounded}. 

First, assume that $A$ satisfies \eqref{eq:boundedunboundeddefinitiona}. 
It follows by Theorem  \ref{main:IJREGULARboundedUnbounded} that \ref{item:T1flat}, \ref{item:T4}, and \ref{item:T5} hold. 
In addition, \ref{item:T1flat} implies \ref{item:T3natural} by Lemma \ref{lem:convergenceoperator} and Corollary \ref{cor:finitedimensionalVSextremepoints}; and condition \ref{item:T6flat} holds by Lemma \ref{lem:finitedimensionalnormkjfdhgd}. We conclude by Corollary \ref{cor:nontrivialJFin} that $A$ satisfies \ref{item:T6}. 

Conversely, assume that \ref{item:T1flat}, \ref{item:T4}, and \ref{item:T6} hold. Then \ref{item:T6} implies \ref{item:T5}. 
As before, \ref{item:T1flat} implies \ref{item:T3natural}, so that $A$ satisfies \ref{item:T3} by Lemma \ref{lem:convergenceoperator}. The conclusion follows by Theorem \ref{main:IJREGULARboundedUnbounded}.
\end{proof}

\begin{proof}
[Proof of Theorem \ref{main:IJREGULARUnboundedBounded}] 
Suppose for the sake of contradiction that such a matrix $A$ exists. 
Since $A \in (c_{00}(X,\mathcal{I}), \ell_\infty(Y))$, it follows by Theorem \ref{thm:c00ellinfty} that 
%
%
%
$$
\exists k_1 \in \omega, \forall k\ge k_1, \forall n \in \omega, \quad A_{n,k}=0.
$$

At this point, 
since $A$ is also $(\mathcal{I},\mathcal{J})$-regular with respect to $T$, then \ref{item:T1}-\ref{item:T5} hold by Theorem \ref{main:IJREGULAR}. 
In addition, $T$ is a nonzero linear operator, hence there exists (a nonzero) $x \in X$ such that $y:=Tx\neq 0$, and pick $\bm{x}:=(x,x,\ldots)$.  
Since $A$ satisfies \eqref{eq:boundedunboundeddefinition}, we obtain that 
$\mathcal{J}\text{-}\lim A\bm{x}=y$. 
However, it follows by condition \ref{item:T5} that 
$$
\mathcal{J}\text{-}\lim\nolimits_n A_n\bm{x}=
\mathcal{J}\text{-}\lim\nolimits_n \sum\nolimits_{k<k_1}A_{n,k}x=0,
$$
providing the desired contradiction. 
\end{proof}

\begin{proof}
[Proof of Theorem \ref{thm:cIc_0bJ}]
This is an immediate consequence of Theorem \ref{thm:c00ellinfty}.
%
%
\end{proof}

\begin{proof}
[Proof of Theorem \ref{main:IJREGULARUnboundedUnbounded}] 
First, suppose that \ref{item:T1flat}, \ref{item:T3sharp}, \ref{item:T4}, and \ref{item:T5sharp} hold for some $k_0 \in \omega$. 
Fix a sequence $\bm{x}$ such that $\mathcal{I}\text{-}\lim \bm{x}=\eta$. 
By \ref{item:T3sharp}, $A\bm{x}$ is well defined. Moreover, since \ref{item:T5sharp} implies \ref{item:T2flat}, we can pick $\kappa$ and $J_1 \in \mathcal{J}^\star$ as in the proof of Theorem \ref{thm:maddoxmain}, and define $E:=\{k \in \omega: \|x_k-\eta\|\ge 1\} \in \mathcal{I}$. 
Using \ref{item:T5sharp}, there exists $J_2 \in \mathcal{J}^\star$ such that $\left\|\sum_{k \in E, k\ge k_0}A_{n,k}x_k\right\|\le 1$ for all $n \in J_2$.
At this point, it follows by \ref{item:T1flat} that
\begin{displaymath}
\begin{split}
\|A_n\bm{x}\|&\le \left\|\sum\nolimits_{k<k_0}A_{n,k}x_k\right\|
+ \left\|\sum\nolimits_{k\in E, k\ge k_0}A_{n,k}x_k\right\|+ \left\|\sum\nolimits_{k\notin E, k\ge k_0}A_{n,k}x_k\right\|\\
&\le \kappa k_0+1+\sup\{\|x_k\|: k\notin E\}\cdot \sup\nolimits_{t \in J_0}\left\|A_{t, E^c \setminus [0,k_0)}\right\|\\
&\le \kappa k_0+1+(\|\eta\| +1)\cdot \sup\nolimits_{t \in J_0}\left\|A_{t,\ge k_0}\right\|
\end{split}
\end{displaymath}
for all $n \in J$, where $J:=J_0\cap J_1\cap J_2 \in \mathcal{J}^\star$. Therefore $A\bm{x} \in \ell_\infty(Y,\mathcal{J})$. 

We conclude as in the proof of Theorem \ref{main:IJREGULAR} that $A$ satisfies \eqref{eq:unoundedunboundeddefinition}, with the only difference that Inequality \eqref{eq:ddjyhkjfgf} holds for all $n \in J_0$, once we replace $\sup_t\|A_{t,\ge k_0}\|$ with $\sup_{t\in J_0}\|A_{t,\ge k_0}\|$.

\medskip 

Conversely, if $A$ satisfies \eqref{eq:unoundedunboundeddefinition}, then 
\ref{item:T3sharp} holds because the sequence $A\bm{x}$ is well defined for each $\bm{x} \in c(X,\mathcal{I})$; moreover, $A\in (c^b(X,\mathcal{I}), c(Y,\mathcal{J}))$, and it follows by Theorem \ref{main:IJREGULARboundedUnbounded} that conditions  \ref{item:T1flat} and \ref{item:T4} hold. 
Lastly, if $\bm{x}$ is a sequence supported on $\mathcal{I}$, then $\mathcal{J}\text{-}\lim A\bm{x}=0$, which proves \ref{item:T5sharp}. 
\end{proof}

\begin{proof}
[Proof of Theorem \ref{thm:cXcYrapid}]
Thanks to Theorem \ref{main:IJREGULARUnboundedUnbounded}, we have only to show that if $A$ satisfies \eqref{eq:unoundedunboundeddefinition} then \ref{item:T1flat} holds. This follows by Theorem \ref{thm:ctoell(X)} since $A \in (c(X),\ell_\infty(Y,\mathcal{J}))$, with $k_0=f(n)=0$ for all $n \in \omega\setminus J_0$ if each $A_{n,k}$ is bounded.
\end{proof}


\section{Closing remarks} 

We leave as open questions for the reader to check whether Theorem \ref{thm:key} holds for the larger class of rapid$^+$ $P^+$-ideals $\mathcal{J}$, and whether this condition characterizes the latter class of ideals in the same spirit of \cite[Theorem 5.1]{MR4172859}. An analogue question could be asked for Theorem \ref{thm:lorentzmacphail}. 
In addition, it would be interesting to obtain a characterization of the matrix class $(c(X,\mathcal{Z}), c(Y,\mathcal{Z}))$, analogously to Theorem \ref{thm:cXcYrapid}. 

\subsection*{Acknowledgments} 
The author is grateful to 
PRIN 2017 (grant 2017CY2NCA) for financial support. He is very thanksful to an anonymous referee for a careful and constructive reading of the manuscript. 

%


\bibliographystyle{amsplain}

\bibliography{core}

\end{document}